\documentclass[review]{elsarticle}

\usepackage{lineno,hyperref}
\modulolinenumbers[5]

\usepackage{amsmath}
\usepackage{amssymb}
\usepackage{amsthm}
\usepackage{algorithm,algpseudocode,float}
\usepackage{booktabs}
\usepackage{comment}
\usepackage{dsfont}
\usepackage{enumerate}
\usepackage{eucal}
\usepackage{float}
\usepackage{graphicx}

\usepackage{indentfirst}
\usepackage{lipsum}
\usepackage{subcaption}

\usepackage{lipsum}
\makeatletter
\def\ps@pprintTitle{%
 \let\@oddhead\@empty
 \let\@evenhead\@empty
 \def\@oddfoot{}%
 \let\@evenfoot\@oddfoot}
\makeatother

\makeatletter
\newenvironment{breakablealgorithm}
  {% \begin{breakablealgorithm}
   \begin{center}
     \refstepcounter{algorithm}% New algorithm
     \hrule height.8pt depth0pt \kern2pt% \@fs@pre for \@fs@ruled
     \renewcommand{\caption}[2][\relax]{% Make a new \caption
       {\raggedright\textbf{\ALG@name~\thealgorithm} ##2\par}%
       \ifx\relax##1\relax % #1 is \relax
         \addcontentsline{loa}{algorithm}{\protect\numberline{\thealgorithm}##2}%
       \else % #1 is not \relax
         \addcontentsline{loa}{algorithm}{\protect\numberline{\thealgorithm}##1}%
       \fi
       \kern2pt\hrule\kern2pt
     }
  }{% \end{breakablealgorithm}
     \kern2pt\hrule\relax% \@fs@post for \@fs@ruled
   \end{center}
  }
\makeatother

\makeatletter
\def\@author#1{\g@addto@macro\elsauthors{\normalsize%
    \def\baselinestretch{1}%
    \upshape\authorsep#1\unskip\textsuperscript{%
      \ifx\@fnmark\@empty\else\unskip\sep\@fnmark\let\sep=,\fi
      \ifx\@corref\@empty\else\unskip\sep\@corref\let\sep=,\fi
      }%
    \def\authorsep{\unskip,\space}%
    \global\let\@fnmark\@empty
    \global\let\@corref\@empty  %% Added
    \global\let\sep\@empty}%
    \@eadauthor={#1}
}
\makeatother

\newtheorem{mydef}{Definition}
\newtheorem{theorem}{Theorem}
\newtheorem{corollary}{Corollary}[theorem]

\begin{document}

\begin{frontmatter}

\title{A New Kalman Filter Model for Nonlinear Systems Based on Ellipsoidal Bounding}

\author{Ligang Sun\fnref{fn1}}
%\ead{ligang.sun@gih.uni-hannover.de}
\fntext[fn1]{Geodetic Institute, Leibniz Universit\"at Hannover\\
Email address: \{ligang.sun, alkhatib, kargoll, neumann\}@gih.uni-hannover.de}

\author{Hamza Alkhatib\fnref{fn1}}

\author{Boris Kargoll\fnref{fn1}}

\author{Vladik Kreinovich\fnref{fn2}}
\fntext[fn2]{Department of Computer Science, University of Texas at El Paso\\
Email address: vladik@utep.edu}

\author{Ingo Neumann\fnref{fn1}}

%\fntext[myfootnote]{Leibniz University Hannover}


%\ead{ligang.sun@gih.uni-hannover.de}


\begin{abstract}
In this paper, a new filter model called set-membership Kalman filter for nonlinear state estimation problems was designed, where both random and unknown but bounded uncertainties were considered simultaneously in the discrete-time system. The main loop of this algorithm includes one prediction step and one correction step with measurement information, and the key part in each loop is to solve an optimization problem. The solution of the optimization problem produces the optimal estimation for the state, which is bounded by ellipsoids. The new filter was applied on a highly nonlinear benchmark example and a two-dimensional simulated trajectory estimation problem, in which the new filter behaved better compared with extended Kalman filter results. Sensitivity of the algorithm was discussed in the end.
\end{abstract}


\end{frontmatter}

%\linenumbers

%\maketitle

\section{Introduction}

State estimation is applicable to virtually all areas of engineering and science. Any discipline that is concerned with the mathematical modeling of its systems is a likely candidate for state estimation. This includes electrical engineering, mechanical engineering, chemical engineering, aerospace engineering, robotics, dynamical systems' control and many others. Nonlinear filtering can be a difficult and complex subject in the field of state estimation. It is certainly not as mature, cohesive, or well understood as linear filtering. There is still a lot of room for advances and improvement in nonlinear estimation techniques.

%The extended Kalman filter(EKF) and particle filter are two widespread methods in nonlinear filtering field. The extended Kalman filter is the most widely applied state estimation algorithm for nonlinear system. However, the EKF often gives unreliable estimates if the system nonlinearities are severe. Based on Bayesian state estimation and Monte Carlo method, the particle filter can give better estimation for both linear and nonlinear system, but the computational effort is often a bottleneck to its implementation.

The optimal state estimation problem can be summarized as follows: given a mathematical model of a real system, and allowing some state perturbations and noise corrupted measurements, the state of the real system has to be estimated \cite{le2013zonotopic}. The estimation usually bases on the solving of an optimization problem, the estimated result relies on the assumptions made on uncertainties. Developed in the past hundreds years, the stochastic state estimation techniques are most widely applied in the real world. This approach bases on the probabilistic assumptions of the uncertainties in the system, such as Kalman filter \cite{kalman1960new} and extended Kalman filter (EKF) \cite{smith1962application,mcelhoe1966assessment} where uncertain parts (usually noise) in the system are assumed to have certain probability distribution (usually Gaussian distribution).

However, in many cases these probability distributions could be questionable, especially when the real process generating the data are complex so that only simplified models can be practically used in the estimation process \cite{milanese1996optimal}. There is another interesting approach, referred to set-membership uncertainty state estimation. Developed since 1960s \cite{witsenhausen1968sets,schweppe1968recursive,schweppe1973uncertainty}, this approach assumes that the uncertainty is unknown but bounded (UBB). No further assumption was made except for its membership of a given bound. Under this assumption, the optimal estimated state, noisy measurements and uncertainty are in some compact sets, respectively. This new technique is more appropriate in many cases where the bounded description is more realistic than stochastic distributed hypothesis. Classified by the geometrical representations, there are four major methods to bound the uncertainty, which are polytopes \cite{vicino1996sequential,walter1989exact}, ellipsoids \cite{bertsekas1971minimax,polyak2004ellipsoidal,durieu2001multi}, zonotopes \cite{combastel2005state,alamo2005guaranteed,althoff2009safety,le2013zonotopic,schon2005using} and intervals \cite{kreinovich2013computational,ferson2007experimental,kutterer2011recursive}. Polytope can be used to obtain better estimated accuracy, however, one major drawback is its computation load in multi-dimensional nonlinear systems, especially to zonotope. Ellipsoid has been widely used due to its simplicity of propagation, but the Minkovski sum of two ellipsoids is not an ellipsoid anymore, therefore the prorogation of its related algorithm requires solving an optimization problem.

In this paper, a new filter model called set-membership Kalman filter (SKF) for nonlinear systems was designed, in which both random and set-membership uncertainties were considered at the same time. This work extends Benjamin Noack's previous work in his PhD dissertation \cite{noack2014state}, where the linear case was discussed sufficiently. The novel SKF takes UBB uncertainties into account in both process equation and measurement equation, therefore it has a better uncertainty measures. It also keeps the recursive framework of random uncertainties from Kalman filter, thus the advantages of KF are reserved during the prorogation process. A better estimation under these more reliable assumptions is calculated based on solving an optimization problem in each step.

Section 2 gives mathematics preliminaries and dynamical system which would be considered later. Section 3 shows the detailed derivation of this new filter model. Section 4 is the algorithm in a practical form. Section 5 demonstrates how this new filter model works and shows that the SKF behaves better than EKF in some cases. The last section is the conclusion and future work.

\section{Mathematical Model}

\subsection{Preliminaries}
The following definitions, theorems and corollaries are required for the derivation of the new filter model. The detailed proofs were given in \cite{durieu2001multi}.
%\subsection{High-dimensional Ellipsoids}

\begin{mydef}\label{def_ellip}
Given $S$ a positive-definite matrix, denoted by $S>0$, a bounded ellipsoid $\mathcal{E}$ in $\mathds{R}^n$ with nonempty interior is defined as
\begin{equation}
\mathcal{E}=\mathcal{E}(c,S)=\{x\in\mathds{R}^n|(x-c)^TS^{-1}(x-c)\leq 1, S>0\}
\end{equation}
where $c\in\mathds{R}^n$ is called the center of the ellipsoid $\mathcal{E}$, and $S$ is the shape matrix which is positive-definite and specifies the size and orientation of the ellipsoid.
\end{mydef}

\begin{mydef}\label{def_min}
In geometry, the Minkowski sum is an operation of two sets $A$ and $B$ in Euclidean space $\mathds{R}^n$, which is defined by adding each vector in $A$ to each vector in $B$, i.e.,
\begin{equation}
A\oplus B=\{a+b|a\in A, b\in B\}.
\end{equation}
\end{mydef}

%Consider an ellipsoid $\mathcal{E}=\mathcal{E}^+(c,S)$ described in the form of \eqref{def_ellip1}.

Given $K$ ellipsoids of $\mathds{R}^n$
\begin{equation}
\mathcal{E}_k=\mathcal{E}(c_k,S_k), (k=1,2,\dots,K)
\end{equation}
their Minkowski sum is
\begin{equation}
\mathcal{U}_K=\sum_{k=1}^{K}\mathcal{E}_k,
\end{equation}
which is not an ellipsoid anymore but still a convex set.

Denote the problem of finding the smallest ellipsoid (under the criterion of matrix trace) containing the Minkowski sum of the $K$ ellipsoids as Problem $\mathrm{T^+}$:
\begin{equation}
\mathcal{E}^*=\underset{\mathcal{U}_K\subset\mathcal{E}}{\arg\min}\,\mathrm{tr} S\quad\mathrm{(Problem\,T^+)},
\end{equation}
and from \cite{durieu2001multi}, this ellipsoid $\mathcal{E}^*$ exists and is unique.

\begin{theorem}\label{thm4.1}
The center of the optimal ellipsoid $\mathcal{E}^*$ for Problem $T^+$ is given by
\begin{equation}\label{e_center}
c^*=\sum_{k=1}^{K}c_k
\end{equation}
\end{theorem}

\begin{theorem}\label{thm4.2}
Let $\mathcal{D}$ be the convex set of all vectors $\alpha\in\mathds{R}^K$ with all $\alpha_k>0$ and $\sum_{k=1}^{K}\alpha_k=1$. For any $\alpha\in\mathcal{D}$, the ellipsoid $\mathcal{E}_\alpha=\mathcal{E}^+(c^*, S_\alpha)$, with $c^*$ defined by \eqref{e_center} and
\begin{equation}
S_\alpha=\sum_{k=1}^{K}\alpha_k^{-1}S_k,
\end{equation}
contains $\mathcal{U}_K$.
\end{theorem}

\begin{corollary}\label{cor4.2}
Special case of Theorem (2.2). When $K=2$, we have $\alpha_1+\alpha_2=1$, the $S_\alpha$ can be rewritten as
\begin{equation}
S_{\alpha}=\frac{1}{\alpha_1}S_1+\frac{1}{\alpha_2}S_2=(1+\frac{1}{\beta})S_1+(1+\beta)S_2
\end{equation}
where $\beta$ can be any nonnegative real number.
\end{corollary}
\begin{proof}
Let $\alpha_2=\frac{1}{1+\beta}$, $\beta\geq 0$ one can easily get above result.
\end{proof}

\begin{theorem}\label{thm4.4}
In the family $\mathcal{E}_\alpha=\mathcal{E}^+(c^*,S_\alpha)$, the minimal-trace ellipsoid containing the sum of the ellipsoids $\mathcal{E}_k=\mathcal{E}^+(c_k,S_k), k=1,2,\dots,K$ is obtained for
\begin{equation}
S_{\alpha^*}=\left(\sum_{k=1}^{K}\sqrt{\mathrm{tr}S_k}\right)\left(\sum_{k=1}^{K}S_k\sqrt{\mathrm{tr}S_k}\right)
\end{equation}
\end{theorem}

\begin{corollary}
Special case of Theorem (2.3). When $K=2$, we have
\begin{equation}
S_{\alpha^*}=(1+\frac{1}{\beta^*})S_1+(1+\beta^*)S_2
\end{equation}
where $\beta^*=\sqrt{\frac{\mathrm{tr}S_1}{\mathrm{tr}S_2}}$.
\end{corollary}

\subsection{Dynamical System}

Consider the following nonlinear dynamical system:

\begin{gather}
x_{k+1}=f_k(x_k,u_k,w_k,a_{1,k},a_{2,k},...,a_{I,k}) \label{system1}\\
y_k=h_k(x_k,v_k,b_k) \label{system2}
\end{gather}
where $x_k$ is a $n$-dimensional state vector, $u_k$ is the known input vector, $w_k\sim \textrm{N}(0,C_k^u)$ is a Gaussian system noise with covariance matrix $C_k^u$, $a_{i,k}\in\mathcal{E}(0,S_{ik}^u)$ is the unknown but bounded perturbation with shape matrix $S_{ik}^u$.  $i=1,2,\dots, I.$ denotes the $i$th set-membership perturbation in the prediction equation. $v_k\sim \textrm{N}(0,C_k^z)$ is the a Gaussian measurement noise with covariance matrix $C_k^z$, and $b_k\in \mathcal{E}(0, S_k^z)$ is the unknown but bounded perturbation with shape matrix $S_k^z$. In this literature, $u$ and $z$ in the parameters denote they are relative to system equation and measurement equation, respectively. All the notations above represent the information at time $k$.

The following Fig. \ref{diagram} shows an estimated schematic diagram via set-membership Kalman filter in 2D case \cite{althoff2009safety}. Different with standard Kalman filter, where the output is usually an gaussian distribution and the mean of the distribution was regarded as the estimated point, in set-membership Kalman filter, a set containing all the mean values of possible distributions was put out.

\begin{comment}
\begin{subequations}\label{system0}
\begin{align}
\label{system1}
x_{k+1}&=f_k(x_k,u_k,w_k,a_{1,k},a_{2,k},...,a_{I,k})\\ \label{system2}
y_k&=h_k(x_k,v_k,b_k)
\end{align}
\end{subequations}
\end{comment}

\begin{figure}[H]
\centering
\includegraphics[width=.5\linewidth]{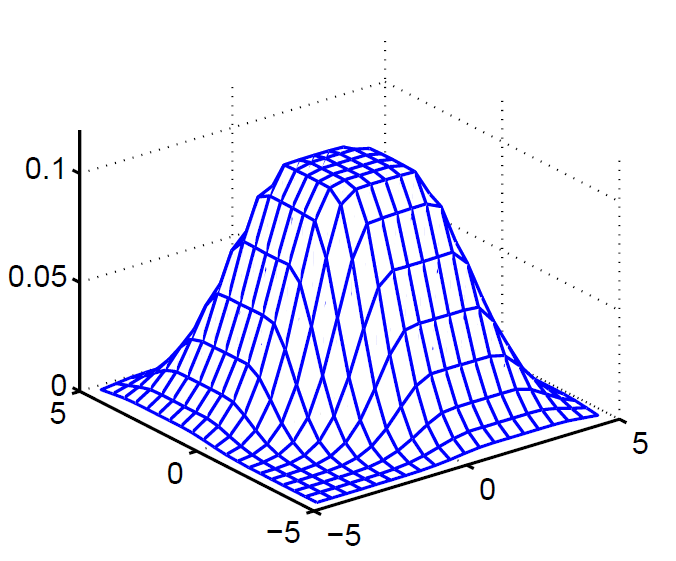}
\caption{Schematic diagram of 2D estimated result under SKF}
\label{diagram}
\end{figure}

\subsection{Linearization}

Recall the process of EKF, linearization is the first step in estimation for nonlinear dynamical systems. Perform Taylor series expansion for system equation \eqref{system1} around the point $(x_k=\hat{x}_k^+,u_k=u_k,w_k=0,a_{i,k}=0, 1\leq i\leq I)$:
\begin{equation}
\begin{split}
x_{k+1}=&f_k(\hat{x}_k^+,u_k,0,0)+\frac{\partial f_k}{\partial x_k}\bigg|_{(\hat{x}_k^+,u_k,0,0)}(x_k-\hat{x}_k^+)+\frac{\partial f_k}{\partial w_k}\bigg|_{(\hat{x}_k^+,u_k,0,0)}w_k\\
&+\sum_{i=1}^{I}\frac{\partial f_k}{\partial a_i}\bigg|_{(\hat{x}_k^+,u_k,0,0)}a_{i,k}+\cdots\\
\approx& f_k(\hat{x}_k^+,u_k,0,0)+F_{x,k}(x_k-\hat{x}_k^+)+F_{w,k}w_k+\sum_{i=1}^{I}F_{ai,k}a_{i,k}\\
=& F_{x,k}x_k+[f_k(\hat{x}_k^+,u_k,0,0)-F_{x,k}\hat{x}_k^+]+F_{w,k}w_k+\sum_{i=1}^{I}F_{ai,k}a_{i,k}\\
=& F_{x,k}x_k+\tilde{u}_k+F_{w,k}w_k+\sum_{i=1}^{I}F_{ai,k}a_{i,k}.
\end{split}
\end{equation}
Here $\tilde{u}_k=f_k(\hat{x}_k^+,u_k,0,0)-F_{x,k}\hat{x}_k^+$.

Take Taylor series expansion for measurement equation \eqref{system2} around point $(x_k=\hat{x}_k^-,v_k=0,b_k=0)$:
\begin{equation}\label{zk}
\begin{split}
y_k=&h_k(\hat{x}_k^-,0,0)+\frac{\partial h_k}{\partial x_k}\bigg|_{(\hat{x}_k^-,0,0)}(x_k-\hat{x}_k^-)+\frac{\partial h_k}{\partial v_k}\bigg|_{(\hat{x}_k^-,0,0)}v_k\\
&+\frac{\partial h_k}{\partial b_k}\bigg|_{(\hat{x}_k^-,0,0)}b_k+\cdots\\
\approx& h_k(\hat{x}_k^-,0,0)+H_{x,k}(x_k-\hat{x}_k^-)+H_{v,k}v_k+H_{b,k}b_k\\
%= & H_{x,k}x_k+[h_k(\hat{x}_k^-,0,0)-H_{x,k}\hat{x}_k^-]+H_{v,k}v_k+H_{b,k}b_k\\
=& H_{x,k}x_k+\tilde{z}_k+H_{v,k}v_k+H_{b,k}b_k.
\end{split}
\end{equation}
Here $\tilde{z}_k=h_k(\hat{x}_k^-,0,0)-H_{x,k}\hat{x}_k^-$. $\tilde{z}_k=0$ if measurement equation is linear.

Then we get the a linearized system for the original system \eqref{system1} and \eqref{system2}.
\begin{gather}
x_{k+1}=F_{x,k}x_k+\tilde{u}_k+F_{w,k}w_k+A_k \label{lsystem1}\\
y_k=H_{x,k}x_k+\tilde{z}_k+H_{v,k}v_k+H_{b,k}b_k \label{lsystem2}
\end{gather}
where $A_k=\sum_{i=1}^{I}F_{ai,k}a_{i,k}$.

\begin{comment}
\begin{subequations}\label{lsystem0}
\begin{align}
x_{k+1}=F_{x,k}x_k+\tilde{u}_k+F_{w,k}w_k+A_k\label{lsystem1}\\
y_k=H_{x,k}x_k+\tilde{z}_k+H_{v,k}v_k+H_{b,k}b_k\label{lsystem2}
\end{align}
\end{subequations}
where $A_k=\sum_{i=1}^{I}F_{ai,k}a_{i,k}$.
\end{comment}

Both priori estimation $\hat{x}_k^-$ and posteriori estimation $\hat{x}_k^+$ are random variables. Assume that the expectation and covariance matrix of priori estimation $\hat{x}_k^-$ are $\hat{\mu}_k^-$ and $C_k^-$, the expectation and covariance matrix of posteriori estimation $\hat{x}_k^+$ are $\hat{\mu}_k^+$ and $C_k^+$. All the priori expectations $\hat{\mu}_k^-$ form an ellipsoid centered at $\hat{x}_k^{c-}$ with shape matrix $S_k^-$, i.e., $\hat{\mu}_k^-\in\mathcal{E}(\hat{x}_k^{c-},S_k^-)$. Similarly to posteriori expectation we have $\hat{\mu}_k^+\in \mathcal{E}(\hat{x}_k^{c+}, S_k^+)$.

Our objective is to calculate the explicit expressions of $\hat{x}_k^{c-}$, $C_k^-$, $S_k^-$ and $\hat{x}_k^{c+}$, $C_k^+$, $S_k^+$.

\section{Derivation of Set-membership Kalman Filter}

After getting the linearized dynamical system \eqref{lsystem1} and \eqref{lsystem2}, in this section we derive the set-membership Kalman filter model. Conclusions from section 2.1 are required and the results of this section would be summarized into one algorithm in section 4.

\subsection{Prediction}

Assume that the difference between the true state $x_k$ and the posteriori estimations center $\hat{x}_k^{c+}$ contains two components, i.e., the random part and the UBB part:
\begin{equation}\label{diff+}
x_k-\hat{x}_k^{c+}=\tilde{x}_k^{r+}+\tilde{x}_k^{s+}.
\end{equation}
So from last section we can get $\tilde{x}_k^{r+}\sim\textrm{N}(0,C_k^+)$ and $\tilde{x}_k^{s+}\in\mathcal{E}(0,S_k^+)$. And the mean squared error of posteriori estimation is given by
\begin{equation}\label{cov_xe}
\begin{split}
&\mathrm{E}[(x_k-\hat{x}_k^{c+})(x_k-\hat{x}_k^{c+})^T]=\mathrm{E}[(\tilde{x}_k^{r+}+\tilde{x}_k^{s+})(\tilde{x}_k^{r+}+\tilde{x}_k^{s+})^T]\\
=&\mathrm{E}[\tilde{x}_k^{r+}\tilde{x}_k^{r+,T}]+\mathrm{E}[\tilde{x}_k^{s+}\tilde{x}_k^{s+,T}]=C_k^++\tilde{x}_k^{s+}\tilde{x}_k^{s+,T}.
\end{split}
\end{equation}

Recalling EKF we have
\begin{equation}
\hat{x}_{k+1}^-=F_{x,k}\hat{x}_k^++\tilde{u}_k+F_{w,k}w_k+A_k.
\end{equation}

Notice that $F_{w,k}w_k\sim\mathrm{N}(0,F_{w,k}C_k^uF_{w,k}^T)$, so for a fixed posteriori estimation $\hat{\mu}_k^+\in \mathcal{E}(\hat{x}_k^{c+},S_k^+)$, the predicted state follows by
\begin{equation}\label{pn}
\hat{x}_{k+1}^-=F_{x,k}\hat{\mu}_k^++\tilde{u}_k+A_k+F_{w,k}w_k\sim\mathrm{N}(F_{x,k}\hat{\mu}_k^++\tilde{u}_k+A_k,F_{w,k}C_k^uF_{w,k}^T\big|\hat{x}_k^+).
\end{equation}
Therefore the expectation of $\hat{x}_{k+1}^-$ would be
\begin{equation}\label{mu-}
\hat{\mu}_{k+1}^-=\mathrm{E}(\hat{x}_{k+1}^-)=F_{x,k}\hat{\mu}_k^++\tilde{u}_k+A_k,
\end{equation}
which forms a set $\mathcal{E}(\hat{x}_{k+1}^{c-},S_{k+1}^-)$ when $\hat{x}_k^+$ being ergodic in the set $\mathcal{E}(\hat{x}_k^{c+},S_k^+)$.

Without loss of generality we have
\begin{equation}
\hat{x}_{k+1}^{c-}=F_{x,k}\hat{x}_k^{c+}+\tilde{u}_k.
\end{equation}

Then the difference between the true state and the priori estimation center would be
\begin{equation}
\begin{split}
x_{k+1}-\hat{x}_{k+1}^{c-}&=F_{x,k}(x_k-\hat{x}_k^{c+})+F_ww_k+A_k\\
&=F_{x,k}(\tilde{x}_k^{r+}+\tilde{x}_k^{s+})+F_ww_k+A_k.
\end{split}
\end{equation}

Consider its covariance matrix we have
\begin{equation}\label{cov_predict}
\begin{split}
&\mathrm{E}[(x_{k+1}-\hat{x}_{k+1}^{c-})(x_{k+1}-\hat{x}_{k+1}^{c-})^T]\\
=&\mathrm{E}\{[F_{x,k}(\tilde{x}_k^{r+}+\tilde{x}_k^{s+})+F_ww_k+A_k]\cdot[F_{x,k}(\tilde{x}_k^{r+}+\tilde{x}_k^{s+})+F_ww_k+A_k]^T\}\\
%=&\mathrm{E}\{[F_{x,k}(\tilde{x}_k^{r+}+\tilde{x}_k^{s+})(\tilde{x}_k^{r+}+\tilde{x}_k^{s+})^TF_{x,k}^T]+F_{x,k}(\tilde{x}_k^{r+}+\tilde{x}_k^{s+})w_k^TF_{w,k}^T\\
%&+F_{x,k}(\tilde{x}_k^{r+}+\tilde{x}_k^{s+})A_k^T+F_{w,k}w_k(\tilde{x}_k^{r+}+\tilde{x}_k^{s+})^TF_{w,k}^T+F_{w,k}w_kw_k^TF_{w,k}^T\\
%&+F_{w,k}w_kA_k^T+A_k(\tilde{x}_k^{r+}+\tilde{x}_k^{s+})^TF_{x,k}^T+A_kw_k^TF_{w,k}^T+A_kA_k^T\}\\
=&F_{x,k}\mathrm{E}[(\tilde{x}_k^{r+}+\tilde{x}_k^{s+})(\tilde{x}_k^{r+}+\tilde{x}_k^{s+})^T]F_{x,k}^T+F_{x,k}\tilde{x}_k^{s+}A_k^T\\
&+F_{w,k}\mathrm{E}(w_kw_k^T)F_{w,k}^T+A_k\hat{x}_k^{s+,T}+A_kA_k^T\\
%=&F_{x,k}\mathrm{E}(\tilde{x}_k^{r+}\tilde{x}_k^{r+,T})F_{x,k}^T+F_{x,k}(\tilde{x}_k^{s+}\tilde{x}_k^{s+,T})F_{x,k}^T+F_{x,k}\tilde{x}_k^{s+}A_k^T\\
%&+F_{w,k}C_k^uF_{w,k}^T+A_k\tilde{x}_k^{s+,T}F_{x,k}^T+A_kA_k^T\\
=&F_{x,k}C_k^+F_{x,k}^T+F_{w,k}C_k^uF_{w,k}^T+(F_{x,k}\tilde{x}_k^{s+}+A_k)(F_{x,k}\tilde{x}_k^{s+}+A_k)^T.
\end{split}
\end{equation}

Compared to equation \eqref{cov_xe}, we find that the predicted random uncertainty can be represented by
\begin{equation}\label{C-}
C_{k+1}^-=F_{x,k}C_k^+F_{x,k}^T+F_{w,k}C_k^uF_{w,k}^T.
\end{equation}

Notice that a possible posteriori mean value $\hat{x}_k^+\in\mathcal{E}(\hat{x}_k^{c+},S_k^+)$, and
\begin{gather}
A_k=\sum_{i=1}^{I}F_{ai}a_{i,k}, a_{i,k}\in\mathcal{E}(0,S_{i,k}^u)\\
F_{ai}a_{i,k}\in\mathcal{E}(0,F_{ai}S_{i,k}^uF_{ai}^T).
\end{gather}

So
\begin{equation}
A_k\in\sum_{i=1}^I\mathcal{E}(0,F_{ai}S_{i,k}^uF_{ai}^T).
\end{equation}
i.e., $A_k$ is one fixed element of a convex set which is the Minkowski sum of $I$ ellipsoids.

Recalling \eqref{mu-} we have
\begin{equation}\label{mu-1}
\begin{split}
\hat{\mu}_{k+1}^-=&\mathrm{E}(\hat{x}_{k+1}^-)=F_{x,k}\hat{\mu}_k^++\tilde{u}_k+A_k\\
\in&\mathcal{E}(F_{x,k}\hat{x}_k^{c+}+\tilde{u}_k,F_{x,k}S_k^+F_{x,k}^T)\oplus\sum_{i=1}^I\mathcal{E}(0,F_{ai}S_{i,k}^uF_{ai}^T)
\end{split}
\end{equation}

Recalling \ref{thm4.4}, there exists an optimal ellipsoid $\mathcal{E}(c_k^*,S_{\alpha^*,k})$ such that
\begin{equation}
\begin{split}
F_{x,k}\hat{\mu}_k^++\tilde{u}_k+A_k&\in\mathcal{E}(F_{x,k}\hat{x}_k^{c+}+\tilde{u}_k,F_{x,k}S_k^+F_{x,k}^T)\oplus\sum_{i=1}^I\mathcal{E}(0,F_{ai}S_{i,k}^uF_{ai}^T)\\
&\subset\mathcal{E}(\hat{x}_{k+1}^{c-},S_{k+1}^-)
\end{split}
\end{equation}

From \ref{thm4.1} we can get the center of the ellipsoid:
\begin{equation}\label{xc-}
\hat{x}_{k+1}^{c-}=F_{x,k}\hat{x}_k^{c+}+\tilde{u}_k.
\end{equation}

From \ref{thm4.4} we can calculate the shape matrix of the ellipsoid:
\begin{equation}\label{S-}
\begin{split}
S_{k+1}^-=&(\sqrt{\mathrm{tr}(F_{x,k}S_k^+F_{x,k}^T)}+\sum_{i=1}^I\sqrt{\mathrm{tr}(F_{a,i}S_{i,k}^uF_{a,i}^T)})\\
&\cdot(\frac{F_{x,k}S_k^+F_{x,k}^T}{\sqrt{\mathrm{tr}(F_{x,k}S_k^+F_{x,k}^T)}}+\sum_{i=1}^I\frac{F_{a,i}S_{i,k}^uF_{a,i}^T}{\sqrt{\mathrm{tr}(F_{a,i}S_{i,k}^uF_{a,i}^T)}})
\end{split}
\end{equation}

Equation \eqref{C-}, \eqref{xc-} and \eqref{S-} gave us the elicit expressions of $C_k^-$, $\hat{x}_k^{c-}$ and $S_k^-$ respectively.

\subsection{Filtering}

Similar with \eqref{diff+}, here we assume that
\begin{equation}
x_k-\hat{x}_k^{c-}=\tilde{x}_k^{r-}+\tilde{x}_k^{s-}.
\end{equation}
So from last section we can get $\tilde{x}_k^{r-}\sim\textrm{N}(0,C_k^-)$ and $\tilde{x}_k^{s-}\in\mathcal{E}(0,S_k^-)$. And the mean squared error of priori estimation is given by
\begin{equation}\label{cov_xe}
\begin{split}
&\mathrm{E}[(x_k-\hat{x}_k^{c-})(x_k-\hat{x}_k^{c-})^T]=\mathrm{E}[(\tilde{x}_k^{r-}+\tilde{x}_k^{s-})(\tilde{x}_k^{r-}+\tilde{x}_k^{s-})^T]\\
=&\mathrm{E}[\tilde{x}_k^{r-}\tilde{x}_k^{r-,T}]+\mathrm{E}[\tilde{x}_k^{s-}\tilde{x}_k^{s-,T}]=C_k^-+\tilde{x}_k^{s-}\tilde{x}_k^{s-,T}.
\end{split}
\end{equation}
\begin{gather}
\tilde{z}_k=h_k(\hat{x}_k^-,0,0)-H_{x,k}\hat{x}_k^-\\
y_k-\tilde{z}_k=H_{x,k}x_k+H_{v,k}v_k+H_{b,k}b_k.
\end{gather}

Therefore, recalling equations \eqref{zk}, \eqref{lsystem1}, \eqref{lsystem2} and the derivation process in EKF, we also assume
\begin{equation}
\begin{split}
\hat{x}_k^+=&\hat{x}_k^-+K_k[y-h_k(\hat{x}_k^-,0,0)]=\hat{x}_k^{-}+K_k[y-\tilde{z}_k(\hat{x}_k^{-})-H_{x,k}\hat{x}_k^{-}]\\
=&(I-K_kH_{x,k})\hat{x}_k^-+K_k[y-\tilde{z}_k(\hat{x}_k^-)].
\end{split}
\end{equation}

The expectations $\hat{\mu}_k^+$ of posteriori estimations $\hat{x}_k^+$ would be
\begin{equation}\label{mu+}
\hat{\mu}_k^+=\mathrm{E}(\hat{x}_k^+)=(I-K_kH_{x,k})\hat{\mu}_k^-+K_k[y-\tilde{z}_k(\hat{\mu}_k^-)].
\end{equation}

The center of the ellipsoid $\mathcal{E}(\hat{x}_k^{c+},S_k^+)$ would be
\begin{equation}
\begin{split}\label{xc+}
\hat{x}_{k}^{c+}=&\hat{x}_k^{c-}+K_k[y-h_k(\hat{x}_k^{c-},0,0)]=\hat{x}_k^{c-}+K_k[y-\tilde{z}_k(\hat{x}_k^{c-})-H_{x,k}\tilde{x}_k^{c-}]\\
=&(I-K_kH_{x,k})\hat{x}_k^{c-}+K_k[y-\tilde{z}_k(\hat{x}_k^{c-})].
\end{split}
\end{equation}

Subtract $\hat{x}_k^{c+}$ from the true state $x_k$ we get:
\begin{equation}
\begin{split}
x_k-\hat{x}_k^{c+}=&x_k-(I-K_kH_{x,k})\hat{x}_k^{c-}-K_k[y_k-\tilde{z}_k(\hat{x}_k^{c-})]\\
=&x_k-(I-K_kH_{x,k})\hat{x}_k^{c-}-K_k(H_{x,k}x_k+H_{v,k}v_k+H_{b,k}b,k)\\
%=&(I-K_{k}H_{x,k})(x_k-\hat{x}_k^{c-})-K_k(H_{v,k}v_k+H_{b,k}b_k)\\
=&(I-K_{k}H_{x,k})(\tilde{x}_k^{r-}+\tilde{x}_k^{s-})-K_k(H_{v,k}v_k+H_{b,k}b_k).
\end{split}
\end{equation}

So the mean squared error of the posteriori estimation center would be
\begin{equation}\label{38}
\begin{split}
&\mathrm{E}[(x_k-\hat{x}_k^{c+})(x_k-\hat{x}_k^{c+})^T]\\
=&\mathrm{E}\{[(I-K_kH_{x,k})(\tilde{x}_k^{r-}+\tilde{x}_k^{s-})+K_k(H_{v,k}v_k+H_{b,k}b_k)]\cdot\\
&[(I-K_kH_{x,k})(\tilde{x}_k^{r-}+\tilde{x}_k^{s-})+K_k(H_{v,k}v_k+H_{b,k}b_k)]^T\}\\
%=&\mathrm{E}\{[(I-K_kH_{x,k})\tilde{x}_k^{r-}+(I-K_kH_{x,k})\tilde{x}_k^{s-}-K_kH_{v,k}v_k-K_kH_{b,k}b_k]\cdot\\
%&[(I-K_kH_{x,k})\tilde{x}_k^{r-}+(I-K_kH_{x,k})\tilde{x}_k^{s-}-K_kH_{v,k}v_k-K_kH_{b,k}b_k]^T\}\\
%=&(I-K_kH_{x,k})\mathrm{E}(\tilde{x}_k^{r-}\cdot\tilde{x}_k^{r-,T})(I-K_kH_{x,k})^T+K_kH_{v,k}\mathrm{E}(v_k\cdot v_k^T)H_{v,k}^TK_k^T\\
%&+\mathrm{E}\{[(I-K_kH_{x,k})\tilde{x}_k^{s-}-K_kH_{b,k}b_k][(I-K_kH_{x,k})\tilde{x}_k^{s-}-K_kH_{b,k}b_k]^T\}\\
=&(I-K_kH_{x,k})C_k^-(I-K_kH_{x,k})^T+K_kH_{v,k}C_k^zH_{v,k}^TK_k^T\\
&+[(I-K_kH_{x,k})\tilde{x}_k^{s-}-K_kH_{b,k}b_k][(I-K_kH_{x,k})\tilde{x}_k^{s-}-K_kH_{b,k}b_k]^T.
\end{split}
\end{equation}

Compared with equation \eqref{cov_xe}, we get
\begin{equation}\label{cov+}
C_k^+=(I-K_kH_{x,k})C_k^-(I-K_kH_{x,k})^T+K_kH_{v,k}C_k^zH_{v,k}^TK_k^T.
\end{equation}
Similar to Kalman filter (KF), the covariance matrices in the SKF provide us with a measure for uncertainty in our predicted and filtering state estimate, which is a very important feature for various applications of filtering theory since we then know how much to trust our predictions and estimates.

Notice that
\begin{equation}
\hat{\mu}_k^-\in\mathcal{E}(\hat{x}_k^{c-},S_k^-)
\end{equation}
and
\begin{equation}
y_k-\tilde{z}_k=H_{x,k}x_k+H_{v,k}v_k+H_{b,k}b_k\in\mathcal{E}(H_{x,k}x_k+H_{v,k}v_k,H_{b,k}S_k^zH_{b,k}^T).
\end{equation}

So back to equation \eqref{mu+} we have
\begin{equation}
\begin{split}
\hat{\mu}_k^+&=(I-K_kH_{x,k})\hat{\mu}_k^-+K_k(y_k-\tilde{z}_k)\\
&\in (I-K_kH_{x,k})\mathcal{E}(\hat{x}_k^{c-},S_k^-)\oplus K_k\mathcal{E}(H_{x,k}x_k+H_{v,k}v_k,H_{b,k}S_k^zH_{b,k}^T)\\
&=\mathcal{E}[(I-K_kH_{x,k})\hat{x}_k^{c-},(I-K_kH_{x,k})S_k^-(I-K_kH_{x,k})^T]\\
&\oplus\mathcal{E}[K_k(H_{x,k}x_k+H_{v,k}v_k),K_kH_{b,k}S_k^zH_{b,k}^TK_k^T]\subset\mathcal{E}(\hat{x}^{c+}_k,S_k^+),
\end{split}
\end{equation}
where the midpoint is exactly in accordance with our previous assumption \eqref{xc+}:
\begin{equation}
\hat{x}_k^{c+}=(I-K_kH_{x,k})\hat{x}_k^{c-}+K_k[y-\tilde{z}_k(\hat{x}_k^{c-})],
\end{equation}
and from Corollary \ref{cor4.2} we have
\begin{equation}\label{shape+}
S_k^+(\beta)=(1+\frac{1}{\beta})(I-K_kH_{x,k})S_k^-(I-K_kH_{x,k})^T+(1+\beta)K_kH_{b,k}S_k^zH_{b,k}^TK_k^T.
\end{equation}

\subsection{Optimization Problem}
Now comparing to its counterpart in EKF, the only thing left is to derive the new optimal Kalman gain, which should minimize the mean square error of the posteriori estimation.

Here we introduce another parameter $\eta\in[0,1]$ to balance the random uncertainty and set-membership in the dynamical system, and define the following cost function as:
\begin{equation}\label{costf}
J(\beta)=(1-\eta)\mathrm{tr}(C_k^+)+\eta\mathrm{tr}(S_k^+(\beta))
\end{equation}
which represents the global uncertainty of the  posteriori estimation. The new optimal Kalman gain to be found should be used to minimize this cost function in a comprehensive way.

Plugging \eqref{cov_xe} and \eqref{shape+} into \eqref{costf} we get:
\begin{equation}
\begin{split}
J(\beta)=&(1-\eta)\mathrm{tr}[(I-K_kH_{x,k})C_k^-(I-K_kH_{x,k})^T]\\
&+(1-\eta)\mathrm{tr}[K_kH_{v,k}C_k^zH_{v,k}^TK_k^T]\\
&+\eta(1+\frac{1}{\beta})\mathrm{tr}[(I-K_kH_{x,k})S_k^-(I-K_kH_{x,k})^T]\\
&+\eta(1+\beta)\mathrm{tr}(K_kH_{b,k}S_k^zH_{b,k}^TK_k^T)\\
\triangleq &(1-\eta)\mathrm{tr}[(I-K_kH_{x,k})C_k^-(I-K_kH_{x,k})^T]\\
&+(1-\eta)\mathrm{tr}[K_kH_{v,k}C_k^zH_{v,k}^TK_k^T]+\eta(1+\frac{1}{\beta})M+\eta(1+\beta)N.
\end{split}
\end{equation}
where $M$ and $N$ are defined directly from above.

Notice that the cost function $J$ relies on two variables $K_k$ and $\beta$. Firstly we minimize $J$ respect with $\beta$.

Since $M>0$ and $N>0$, therefore
\begin{equation}
(1+\frac{1}{\beta})M+(1+\beta)N=M+N+\frac{1}{\beta}M+\beta N\geq M+N+2\sqrt{MN}.
\end{equation}
When $\frac{1}{\beta}M=\beta N$, i.e., $M=\beta^2N,\beta=\beta_1=\sqrt{\frac{M}{N}}$, we have
\begin{equation}
(1+\frac{1}{\beta_1})M+(1+\beta_1)N=M+N+2\sqrt{MN}=(\sqrt{M}+\sqrt{N})^2
\end{equation}
Therefore we can find the local minimum point of function $J$ with respect to $\beta$:
\begin{equation}\label{G}
\begin{split}
J(\beta_1)=&(1-\eta)\mathrm{tr}[(I-K_kH_{x,k})C_k^-(I-K_kH_{x,k})^T]\\
&+(1-\eta)\mathrm{tr}[K_kH_{v,k}C_k^zH_{v,k}^TK_k^T]+\eta(\sqrt{M}+\sqrt{N})^2.
\end{split}
\end{equation}

Next we calculate the global minimum by taking $K_k$ into account.

Notice that
\begin{equation}
\begin{split}
\frac{\partial\sqrt{M}}{\partial K}=&\frac{1}{2}M^{-\frac{1}{2}}\frac{\partial M}{\partial K}=-\frac{1}{2\sqrt{M}}(I-K_kH_{x,k})(S_k^{-,T}+S_k^-)H_{x,k}^T\\
=&-\frac{1}{\sqrt{M}}(I-K_kH_{x,k})S_k^-H_{x,k}^T
\end{split}
\end{equation}
\begin{equation}
\begin{split}
\frac{\partial\sqrt{N}}{\partial K}=&\frac{1}{2}N^{-\frac{1}{2}}\frac{\partial N}{\partial K}=\frac{1}{2\sqrt{N}}K_kH_{b,k}(S_k^{z,T}+S_k^z)H_{b,k}^T\\
&=\frac{1}{\sqrt{N}}K_kH_{b,k}S_k^zH_{b,k}^T
\end{split}
\end{equation}

Then
\begin{equation}
\begin{split}
\frac{\partial J}{\partial K_k}=&2(1-\eta)(K_kH_{x,k}-I)C_k^-H_{x,k}^T+2(1-\eta)K_kH_{v,k}C_k^zH_{v,k}^T\\
&+2\eta(1+\frac{1}{\beta})(K_kH_{x,k}-I)S_k^-H_{x,k}^T+2\eta(1+\beta)K_kH_{b,k}S_k^zH_{b,k}^T.
\end{split}
\end{equation}

Let$\frac{\partial G_1}{\partial K_k}=0$ and solve for $K_k$, we get an adaptive Kalman gain:
\begin{equation}\label{KG}
\begin{split}
K_k=&[(1-\eta)C_k^-H_{x,k}^T+\eta(1+\frac{1}{\beta})S_k^-H_{x,k}^T]\cdot[(1-\eta)H_{x,k}C_k^-H_{x,k}^T\\
&+(1-\eta)H_{v,k}C_k^zH_{v,k}^T+\eta(1+\frac{1}{\beta})H_{x,k}S_k^-H_{x,k}^T+\eta(1+\beta)H_{b,k}S_k^zH_{b,k}^T]^{-1}
\end{split}
\end{equation}

Now we get the elicit expression of the cost function \eqref{costf} by collecting \eqref{cov+}, \eqref{shape+} and \eqref{KG}. All the procedures in this filtering step rely on the solution of the following optimization problem.

\begin{equation}\label{optimization}
\begin{split}
&\underset{\beta}{\min}\,\,\,J(\beta)\\
&\mathrm{s.t.} \,\,\,\beta\in [0,+\infty) \subset \mathds{R}^1
\end{split}
\end{equation}
where the cost function $J(\beta)$ was defined in \eqref{costf} and the solution of above optimization problem was denoted by $\beta^*$. Putting $\beta^*$ into \eqref{KG}, \eqref{cov+} and \eqref{shape+} and we finished the filtering step.

Here are three remarks about this optimization problem:
\begin{enumerate}[(1)]
\item Problem \eqref{optimization} is a nonlinear programming problem since the objective function \eqref{costf} is nonlinear.
\item Problem \eqref{optimization} is a convex optimization problem \cite{noack2014state}. Therefore, any existing local minimum is a global minimum.
\item Usually, it is hard to solve a nonlinear programming problem due to the constrained equations or inequalities. MATLAB function \emph{fminsearch} is an efficient way to solve the problem \eqref{optimization}. Further, an advanced toolbox \emph{INTLAB} can also be used \cite{Ru99a}.
\end{enumerate}

The parameter $\eta$ was introduced to balance the random uncertainty and set-membership uncertainty. There are three very interesting cases need to be noticed \cite{durieu2001multi}.

When $\eta=\frac{1}{2}$, the stochastic uncertainty and set-membership uncertainty have the same weight and $K(p)$ contains no $\alpha$ in this case. This solution is recommended to users when there is no expert-based information about $\eta$ available.

When $\eta=0$,
\begin{equation}
K_k(\beta)=C_k^-H_{x,k}^T\cdot[H_{x,k}C_k^-H_{x,k}^T+H_{v,k}C_k^zH_{v,k}^T]^{-1}
\end{equation}
which is exactly the Kalman gain in the standard EKF \cite{simon2006optimal}.

When $\eta=1$, the model now only contains set-membership uncertainty. In this case,
\begin{equation}
K_k(\beta)=(1+\frac{1}{\beta})S_k^-H_{x,k}^T\cdot[(1+\frac{1}{\beta})H_{x,k}S_k^-H_{x,k}^T+(1+\beta)H_{b,k}S_k^zH_{b,k}^T]^{-1}.
\end{equation}

\section{Algorithm}

An algorithm for SKF was summarized according to previous derivation.

\hfill

\begin{breakablealgorithm}
  \caption{Set-membership Kalman filter model}
  \begin{algorithmic}[1]

  \State  \textbf{Initialization}:
    \begin{enumerate}[(1)]
    \item Initial state midpoint $\hat{x}_0^{c+}=x_0$.
    \item Initial estimated random covariance matrix $C_0^+$.
    \item Initial estimated set-membership shape matrix $S_0^+$.
    \end{enumerate}

  \For{k=1,2,\dots,K}
    \State \textbf{Input of Prediction Step}:
    \begin{enumerate}[(1)]
    \item Point post-estimation $\hat{x}_k^+$, with estimated covariance $C_k^+$ and shape matrix $S_k^+$.
    \item Nonlinear system model
    \begin{equation}
    x_{k+1}=f_k(x_k,u_k,w_k,a_{1,k},a_{2,k},...,a_{I,k}),
    \end{equation}
    where $w_k\sim\textrm{N}(0,C_k^u)$ and $a_{i,k}\in \mathcal{E}(0,S_{i,k}^u)$, $i=1,2,\dots, I.$
    \item Control input $u_k$, random noise covariance $C_k^u$ and shape matrices $S_{i,k}^u$, $i=1,2,\dots, I.$ for set-membership uncertainty.
    \end{enumerate}

    \State \textbf{Calculation of Prediction Step}:
    \begin{enumerate}[(1)]
\item Computation of error covariance matrix $C_{k+1}^-$ according to
\begin{equation}
C_{k+1}^-=F_{x,k}C_k^+F_{x,k}^T+F_{w,k}C_k^uF_{w,k}^T.
\end{equation}
\item The center of the priori ellipsoid:
\begin{equation}
\hat{x}_{k+1}^{c-}=F_{x,k}\hat{x}_k^{c+}+\tilde{u}_k.
\end{equation}
\item The shape matrix of the priori ellipsoid:
\begin{equation}
\begin{split}
S_{k+1}^-=&(\sqrt{\mathrm{tr}(F_{x,k}S_k^+F_{x,k}^T)}+\sum_{i=1}^I\sqrt{\mathrm{tr}(F_{a,i}S_{i,k}^uF_{a,i}^T)})\\
&\cdot(\frac{F_{x,k}S_k^+F_{x,k}^T}{\sqrt{\mathrm{tr}(F_{x,k}S_k^+F_{x,k}^T)}}+\sum_{i=1}^I\frac{F_{a,i}S_{i,k}^uF_{a,i}^T}{\sqrt{\mathrm{tr}(F_{a,i}S_{i,k}^uF_{a,i}^T)}}).
\end{split}
\end{equation}
The predicted point estimate $x_{k+1}^-$ is characterized by the random error $C_{k+1}^-$ and the set-membership error by $\hat{x}_{k+1}^{c-}$ and $S_{k+1}^-$.
\end{enumerate}

    \State \textbf{Output of Prediction Step}:

    Priori estimated state: $\hat{x}_k^{c-}$, $C_k^-$, and $S_k^-$.

    \State \textbf{Input of Filtering Step}:
    \begin{enumerate}[(1)]
\item Priori or predicted  estimate $\hat{x}_k^-$ with error covariance matrix $C_k^-$ and ellipsoid center $\hat{x}_k^{c-}$ and shape matrix $S_k^-$.
\item Nonlinear measurement model:
\begin{equation}
y_k=h_k(x_k,v_k,b_k),
\end{equation}
where $v_k\sim\textrm{N}(0,C_k^z)$ and $b_k\in \mathcal{E}(0,S_k^z)$.
\item Observation $y_k$, sensor noise with random covariance $C_k^z$ and set-membership shape matrix $S_k^z$.
\item $\tilde{z}_k(\hat{x}_k^-)=h_k(\hat{x}_k^-,0,0)-H_{x,k}\hat{x}_k^-$.
\item Weighting parameter $\eta$.
\end{enumerate}

    \State \textbf{Calculation of Filtering Step}:

\begin{enumerate}[(1)]
\item For given weighting parameter $\eta$, the optimal gain factor $K_k$ is
\begin{equation}
\begin{split}
K_k(\beta)=&[(1-\eta)C_k^-H_{x,k}^T+\eta(1+\frac{1}{\beta})S_k^-H_{x,k}^T]\cdot[(1-\eta)H_{x,k}C_k^-H_{x,k}^T\\
&+(1-\eta)H_{v,k}C_k^zH_{v,k}^T+\eta(1+\frac{1}{\beta})H_{x,k}S_k^-H_{x,k}^T\\
&+\eta(1+\beta)H_{b,k}S_k^zH_{b,k}^T]^{-1}.
\end{split}
\end{equation}
\item Computation of the center of updated estimate $\hat{x}_k^+$ by means of
\begin{equation}
\hat{x}_k^{c+}=(I-K_kH_{x,k})\hat{x}_k^{c-}+K_k[y-\tilde{z}_k(\hat{x}_k^{c-})].
\end{equation}
\item Computation of updated error covariance matrix $C_k^+$ by
\begin{equation}
C_k^+(\beta)=(I-K_kH_{x,k})C_k^-(I-K_kH_{x,k})^T+K_kH_{v,k}C_k^zH_{v,k}^TK_k^T.
\end{equation}
\item Update the shape matrix $S_k^+$ by
\begin{equation}
S_k^+(\beta)=(1+\frac{1}{\beta})(I-K_kH_{x,k})S_k^-(I-K_kH_{x,k})^T+(1+\beta)K_kH_{b,k}S_k^zH_{b,k}^TK_k^T.
\end{equation}
\item The optimal parameter $\beta^*$ can be solved by
\begin{equation}\label{op}
\beta^*=\arg \min\{(1-\eta)\mathrm{tr}[C_k^+(\beta)]+\eta\mathrm{tr}[S_k^+(\beta)]\}.
\end{equation}
The updated point estimate $\hat{x}_k^+$ is characterized by random error characteristic $C_k^+=C_k^+(\beta)$ and set-membership error description $S_k^+=S_k^+(\beta)$. Put $\beta^*$ into above 4 functions to get the optimal output.
\end{enumerate}

    \State \textbf{Output of Filtering Step}:

    Posteriori estimated state: $\hat{x}_k^{c+}$, $C_k^+(\beta^*)$, and $S_k^+(\beta^*)$.

\EndFor

\end{algorithmic}
\end{breakablealgorithm}

\section{Applications}

\subsection{Example 1: Highly Nonlinear Benchmark Example}

Consider the following example:
\begin{gather}
x_{k+1}=\frac{1}{2}x_k+\frac{25x_k}{1+x_k^2}+8\cos[1.2(k-1)]+w_k+a_k,\\
y_k=\frac{1}{20}x_k^2+v_k+b_k.
\end{gather}

where $x_k$ is a scalar, $u_k=8\cos[1.2(k-1)]$ is the input vector, $w_k\sim \textrm{N}(0,1)$ is a Gaussian process noise, $a_{i,k}\in\mathcal{E}(0,9)$ is the unknown but bounded perturbation, in this 1-D case the ellipsoid is the interval $[-3,3]$. $v_k\sim \textrm{N}(0,1)$ is the a Gaussian measurement noise, and $b_k\in \mathcal{E}(0,4)$ is the unknown but bounded perturbation in the interval $[-2,2]$. Initial true state is $x_0=0.1$, initial state estimate as $\hat{x}_0=x_0$, initial estimation covariance matrix is $C_0^+=2$ and initial shape matrix is $S_0^+=1\times10^{-3}$. We used a simulation length of 50 time steps. Weight parameter $\eta=0.5$.

This system was regarded as a benchmark in the nonlinear estimation theory \cite{kitagawa1987non}\cite{gordon1993novel}, and it is usually used to demonstrate the drawbacks of EKF comparing with particle filter \cite{simon2006optimal}. The high degree of nonlinearity in both the process and measurement equations makes this system a very difficult state estimation problem for a Kalman filter. We use this example to show the new SKF behaves better than the traditional first order EKF when some set-membership uncertainties are included in the system.

We repeated this simulation for 100 times. And Fig. \ref{com4} shows the comparison results between SKF and EKF at time $k=25,50,75,100$.

\begin{figure}[H]
\centering
\begin{subfigure}[b]{.49\linewidth}
\includegraphics[width=\linewidth]{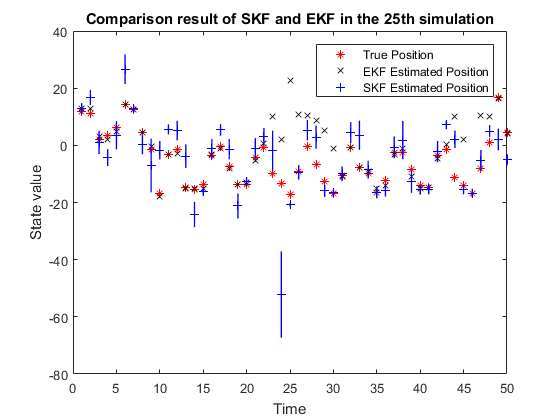}
%\caption{Comparison result in the 25th simulation}\label{fig:mouse}
\end{subfigure}
\begin{subfigure}[b]{.49\linewidth}
\includegraphics[width=\linewidth]{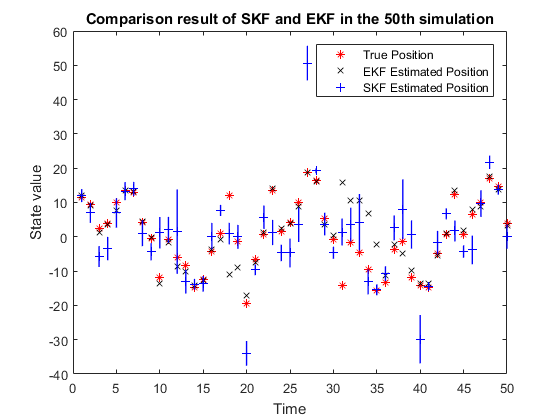}
%\caption{Comparison result in the 50th simulation}\label{fig:mouse}
\end{subfigure}
\begin{subfigure}[b]{.49\linewidth}
\includegraphics[width=\linewidth]{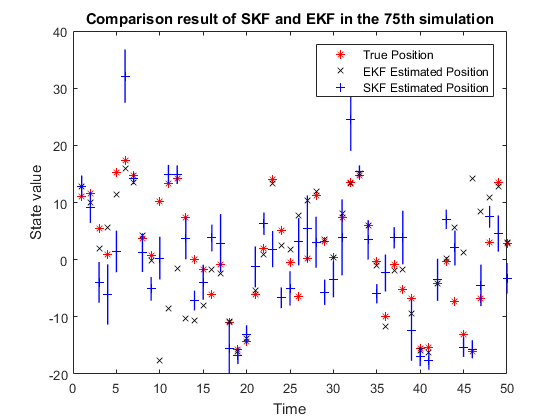}
%\caption{Comparison result in the 75th simulation}\label{fig:gull}
\end{subfigure}
\begin{subfigure}[b]{.49\linewidth}
\includegraphics[width=\linewidth]{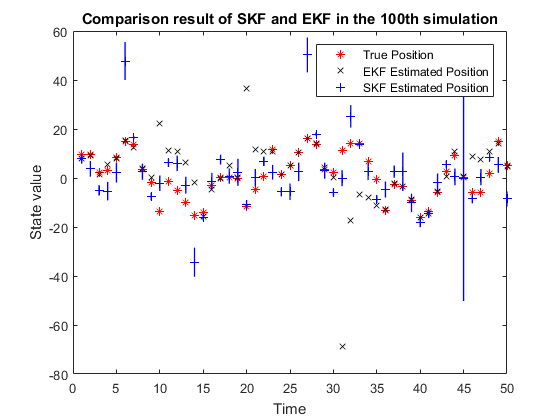}
%\caption{Comparison result in the 100th simulation}\label{fig:tiger}
\end{subfigure}
\caption{Comparison results in 4 simulations}
\label{com4}
\end{figure}

In above figures, the red stars denote the true positions of the state, the black crosses represent the estimated positions via EKF, the blue lines give the estimated ellipsoids (in this 1D case they are intervals) via SKF, and the blue plus signs mark the centers of the output ellipsoids. The center of the ellipsoid given by the new SKF is different with the traditional estimation via EKF, as what we expected, the ellipsoids include the true positions sometimes.

To further compare this new method with EKF, we calculate the distance vectors $d_s, d_e$ of SKF and EKF with the true states, respectively. Each distance vector is $50\times1$ for the total $50$ steps in every simulation. Table \ref{table1} shows the detailed $l_2$ norm comparison of these two distance vectors in $10$ simulations ($k=1,11,...,91$). The second row headed by $\mathrm{S}$ shows the distance error via SKF, and the third row headed by $\mathrm{E}$ shows the counterpart via EKF. We use the $l_2$ norm here as a generic measure of the distance between the estimated data and the true data, but other norms like $l_1$ and $l_{\infty}$ are possible for use. Without loss of generality, we choose the midpoints of these ellipsoids for comparison.

\begin{table}
\scriptsize
\begin{tabular}{|p{0.12cm}|c|c|c|c|c|c|c|c|c|c|}
  \hline
   k & 1 & 11 & 21 & 31 & 41 & 51 & 61 & 71 & 81 & 91 \\
  \hline
  S & 53.59 & 116.38 & 70.38 & 102.95 & 100.16 & 75.81 & 82.11 & 48.71 & 66.30 & 81.55 \\
  \hline
  E & 72.49 & 116.64 & 230.08 & 234.69 & 80.73 & 31.32 & 109.53 & 79.69 & 8.63 & 64.95 \\
  \hline
\end{tabular}
\caption{Comparison of SKF and EKF in 10 simulations}
\label{table1}
\end{table}

In the whole 100-time simulation experiments, the overall $l_2$ norm of the distance vector under SKF is 148.70, with its counterpart in extended Kalman filter 192.29. The new SKF behaved much better than EKF in these 100 simulations. However, this does not mean the SKF is always a better algorithm comparing with EKF, since it is also possible to get opposite results when repeating this experiment. 

\subsection{Example 2: Two-Dimensional Trajectory Estimation}

% Experiment design.

A vehicle moves on a plane with a curved trajectory \cite{alkhatib2008comparison}. The state vector $x=(x,y,v_x,v_y)$ contains positions and velocities of the target, in x-direction and y-direction, respectively. After linearization, we do not consider the acceleration process anymore, and the mathematical model of this vehicle was assumed as following:
\begin{equation}
\mathbf{x}_{k+1}=F_k\mathbf{x}_k+\mathbf{w}_k+\mathbf{a}_k
\end{equation}
where $\mathbf{x}_k=(x_k,y_k,v_{x,k},v_{y,k})$ is the state vector at time $t_k$. The transition matrix $F_k$ is designed by:
\begin{equation}
F_k=\left(
  \begin{array}{cccc}
    1 & 0 & dt & 0 \\
    0 & 1 & 0 & dt \\
    0 & 0 & 1 & 0 \\
    0 & 0 & 0 & 1 \\
  \end{array}
\right).
\end{equation}
$\mathbf{w_k}$ which representing random uncertainty is gaussian with covariance matrix $C_k^u$, and $\mathbf{a_k}$ is the unknown but bounded uncertainty, which was bounded by an ellipsoid with shape matrix $S_k^u$. In total 300 points were observed and time step $dt=0.1$ seconds. The units of time, distance, angle are second, meter and degree, respectively.

In this experiment, two observation stations $S_1=[s_{12}, s_{12}]$ and $S_2=[s_{21}, s_{22}]$ were arranged to make the measurements. Each station measured the distance and the direction angle of the vehicle. Here is the measurement equation:
\begin{equation}
\mathbf{y}_k=h_k(\mathbf{x}_k,v_k,b_k).
\end{equation}
\begin{equation}
\mathbf{y}_k=\left(
  \begin{array}{c}
    d_1 \\
    d_2 \\
    \theta_1 \\
    \theta_2 \\
  \end{array}
\right)=\left(
          \begin{array}{c}
            \sqrt{[x-s_{11}]^2+[y-s_{12}]^2} \\
            \sqrt{[x-s_{21}]^2+[y-s_{22}]^2} \\
            \arctan [(y-s_{12})/(x-s_{11})] \\
            \arctan [(y-s_{22})/(x-s_{21})] \\
          \end{array}
        \right)+\mathbf{v}_k+\mathbf{b}_k
\end{equation}
$\mathbf{v_k}$ which representing random uncertainty is gaussian with covariance matrix $C_k^z$, and $\mathbf{b_k}$ is the unknown but bounded uncertainty, which was bounded by an ellipsoid with shape matrix $S_k^z$.

The initial state, estimated covariance matrix and shape matrix are given by: $x_0=(0,0,0,0)$, $C_0^+=\mathrm{diag}(0.01,0.01,0.01,0.01)$ and $C_0^+=\mathrm{diag}(1\times10^{-6},1\times10^{-6},1\times10^{-6},1\times10^{-6})$.

\begin{comment}
\begin{equation}\label{initial}
x_0=\left(
                  \begin{array}{c}
                   0 \\
                   0 \\
                   0 \\
                   0 \\
                  \end{array}
                \right),
C_0^+=\left(
                    \begin{array}{cccc}
                      0.01 & 0 & 0 & 0 \\
                      0 & 0.01 & 0 & 0 \\
                      0 & 0 & 0.01 & 0 \\
                      0 & 0 & 0 & 0.01 \\
                    \end{array}
                  \right),
S^+_0=1\times10^{-6}\left(
      \begin{array}{cccc}
        1 & 0 & 0 & 0 \\
        0 & 1 & 0 & 0 \\
        0 & 0 & 1 & 0 \\
        0 & 0 & 0 & 1 \\
      \end{array}
    \right).
\end{equation}
\end{comment}

The initial covariance matrices in process equation and measurement equation are given by:
\begin{comment}
\begin{gather}
C_0^u=[0.0033,0,0.005,0;0,0.0033,0,0.005;0.005,0,0.01,0;0,0.005,0,0.01]\\
C_0^z=\mathrm{diag}[0.005^2,0.005^2,0.005^2,0.005^2].
\end{gather}
\end{comment}
\begin{equation}\label{Q}
C^u_0=\left(
                  \begin{array}{cccc}
                    0.0033 & 0 & 0.005 & 0 \\
                    0 & 0.0033 & 0 & 0.005 \\
                    0.005 & 0 &  0.01 & 0 \\
                    0 & 0.005 & 0 & 0.01 \\
                  \end{array}
                \right),
C^z_0=\left(
                  \begin{array}{cccc}
                    0.005^2 & 0 & 0 & 0 \\
                    0 & 0.005^2 & 0 & 0 \\
                    0 & 0 &  0.005^2 & 0 \\
                    0 & 0 & 0 & 0.005^2 \\
                  \end{array}
                \right).
\end{equation}

The initial shape matrices of set-membership uncertainties in process equation and measurement equation are setting by:
\begin{equation}\label{set_membership_setting}
S^u_0=\left(
      \begin{array}{cccc}
        1^2 & 0 & 0 & 0 \\
        0 & 1^2 & 0 & 0 \\
        0 & 0 & 0.5^2 & 0 \\
        0 & 0 & 0 & 0.5^2 \\
      \end{array}
    \right),
S^z_0=\left(
      \begin{array}{cccc}
        0.01^2 & 0 & 0 & 0 \\
        0 & 0.01^2 & 0 & 0 \\
        0 & 0 & (\frac{\pi}{180})^2 & 0 \\
        0 & 0 & 0 & (\frac{\pi}{180})^2 \\
      \end{array}
    \right).
\end{equation}

Weight parameter $\eta=0.5$.

%2. Eight trajectories results and 100-time error.

Below eight different trajectories were estimated by EKF and SKF from eight different data sets. The following Fig. \ref{8trajectory} shows the estimation results. Red stars mark the true position according to the reference data, black crosses denotes the estimated position via EKF, and the blue plus signs are the geometry centers of the ellipsoids.

\begin{figure}[H]
\centering
\begin{subfigure}[b]{.49\linewidth}
\includegraphics[width=\linewidth]{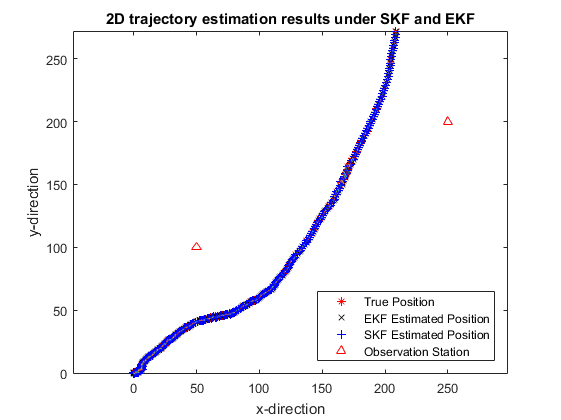}
\caption{Trajectory 1}
\end{subfigure}
\begin{subfigure}[b]{.49\linewidth}
\includegraphics[width=\linewidth]{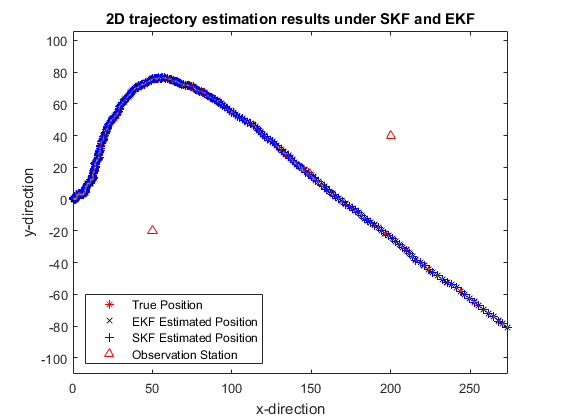}
\caption{Trajectory 2}
\end{subfigure}

\begin{subfigure}[b]{.49\linewidth}
\includegraphics[width=\linewidth]{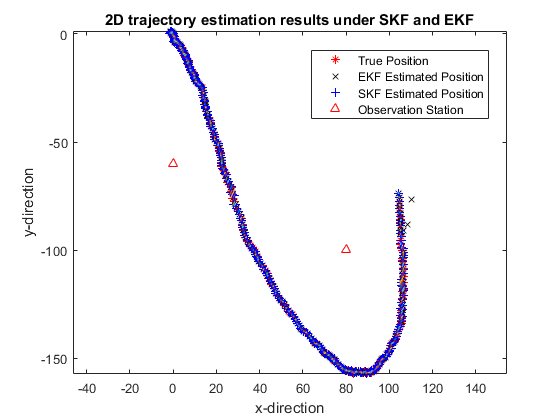}
\caption{Trajectory 3}
\end{subfigure}
\begin{subfigure}[b]{.49\linewidth}
\includegraphics[width=\linewidth]{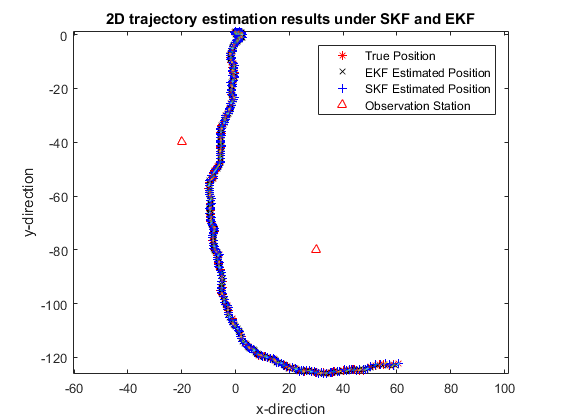}
\caption{Trajectory 4}
\end{subfigure}
\end{figure}

\begin{figure}[H]
\centering
\begin{subfigure}[b]{.49\linewidth}
\includegraphics[width=\linewidth]{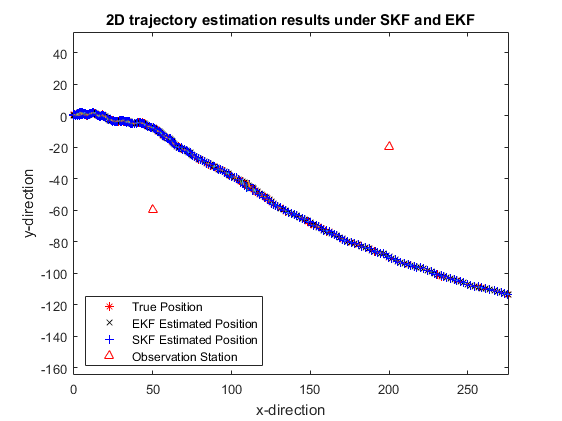}
\caption{Trajectory 5}
\end{subfigure}
\begin{subfigure}[b]{.49\linewidth}
\includegraphics[width=\linewidth]{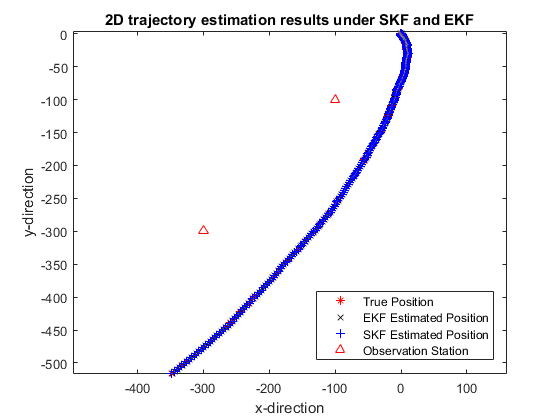}
\caption{Trajectory 6}
\end{subfigure}

\begin{subfigure}[b]{.49\linewidth}
\includegraphics[width=\linewidth]{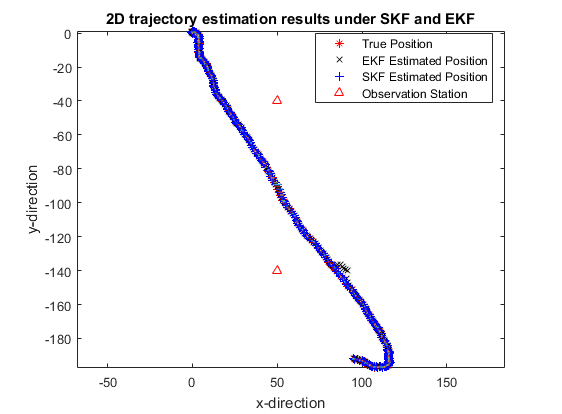}
\caption{Trajectory 7}
\end{subfigure}
\begin{subfigure}[b]{.49\linewidth}
\includegraphics[width=\linewidth]{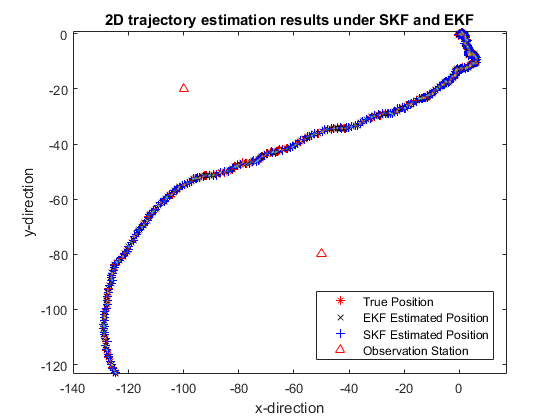}
\caption{Trajectory 8}
\end{subfigure}
\caption{Eight trajectories examples to compare SKF with EKF}
\label{8trajectory}
\end{figure}

One may notice that both EKF and SKF perform well in most part of each trajectory, except that the ellipsoids getting large in the interaction area between the trajectory and the straight line of two stations. Again, we calculate the $l_2$ norm of distance vectors to make further comparison in Table \ref{table2}.

\begin{table}[H]
\centering
\scriptsize
\begin{tabular}{|c|c|c|c|c|c|c|c|c|}
  \hline
   Trajectory & 1 & 2 & 3 & 4 & 5 & 6 & 7 & 8 \\
    \hline
  SKF & 3.22 & 2.88 & 4.36 & 2.47 & 3.70 & 12.37 & 2.83 & 1.91  \\
    \hline
  EKF & 5.37 & 3.62 & 27.81 & 3.82 & 3.67 & 7.79 & 13.31 & 3.27  \\
  \hline
\end{tabular}
\caption{Comparison of SKF and EKF in 8 trajectories}
\label{table2}
\end{table}

To check the estimation errors, we chose Trajectory 5 to repeat for 100 times and then get the following error distribution.

\begin{figure}[H]
\centering
\includegraphics[width=0.9\linewidth]{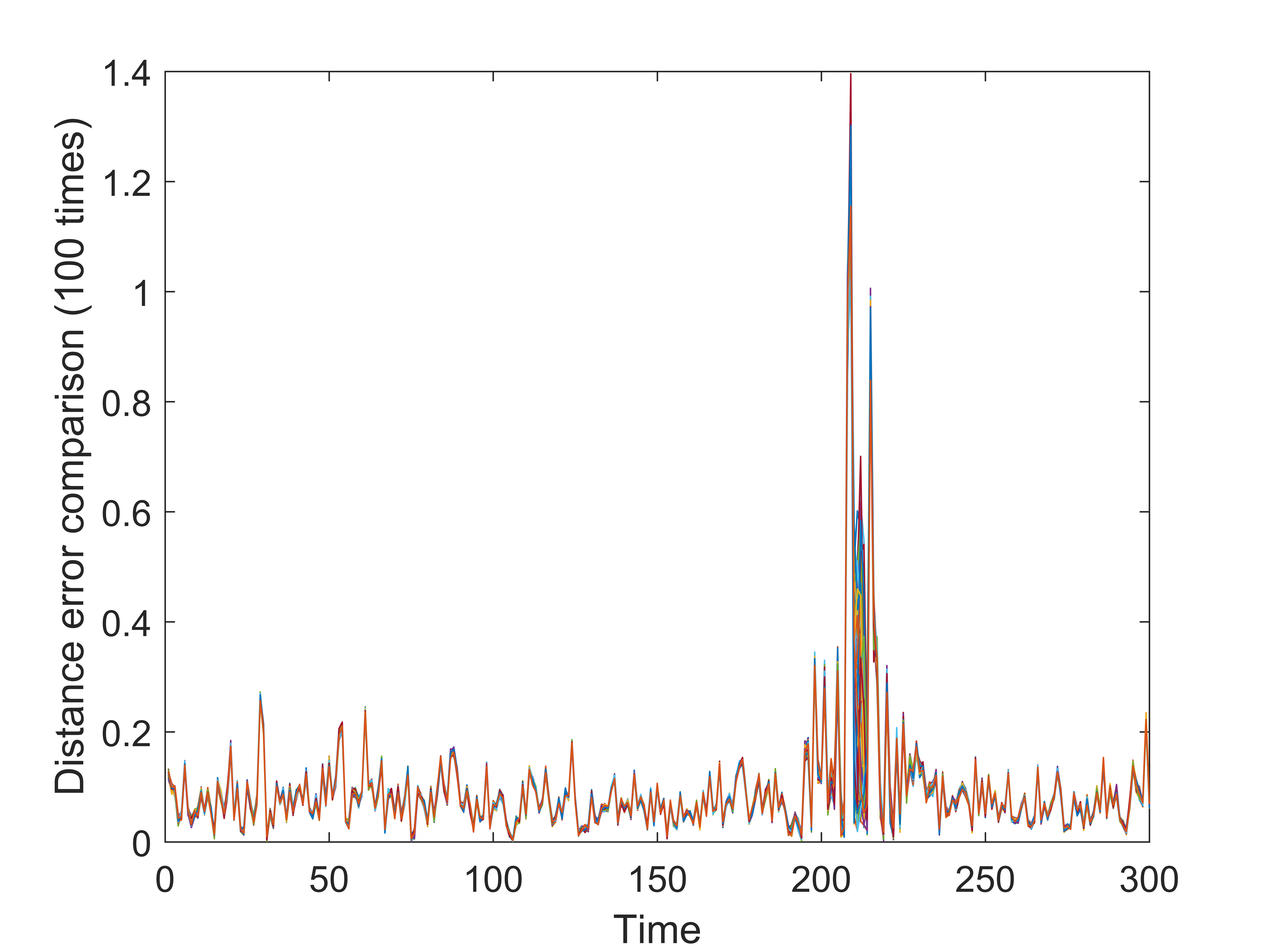}
\caption{Distance Error of SKF in 100 times}
\label{distance_error}
\end{figure}

%3. Strange phenomena and analysis (Ingo).

From above Fig. \ref{distance_error} it is obvious to notice that the estimated error was getting larger when $k\in [200, 225]$, i.e., in Fig. \ref{8trajectory} (e) one may get worse estimation results in the intersection area of the line between the two observation stations and the trajectory of the vehicle. The following Fig. \ref{local_5} shows more local details in the interaction area of Trajectory 5, where the straight line connects the two observation stations. Not only the estimated ellipsoids getting larger in the interaction area, but the semi-major axes of the largest ellipsoid is perpendicular to the straight line.

\begin{figure}[H]
\centering
\includegraphics[width=0.9\linewidth]{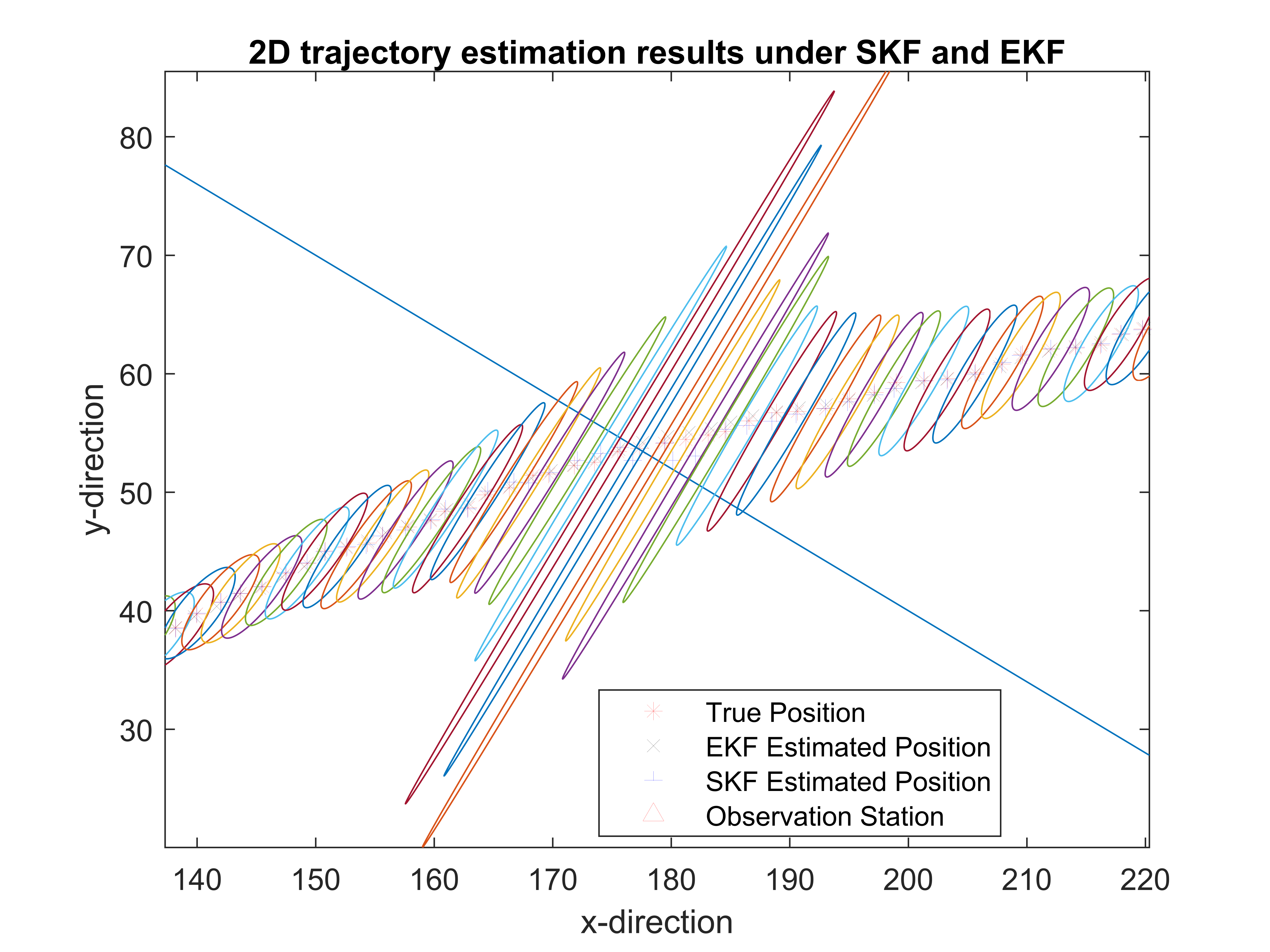}
\caption{Local estimation details near the interaction area}
\label{local_5}
\end{figure}

There exist two major reasons causing this phenomena.

Firstly, the angle set-membership uncertainty played a more significant role in the estimation. From \eqref{set_membership_setting} we notice that in Fig. \ref{8trajectory} Trajectory 5, the set-membership uncertainty of distance is $[-0.01,0.01]$ meter, and its counterpart in angle is $[-1,1]$ degree. The distance between the observation station and the interaction area is at least 85 meters, i.e., the uncertainty caused by angles would be $80\times\frac{\pi}{180}\mathrm{m}=1.4835\mathrm{m}$ (in the vertical direction of the straight line), which is greatly larger than the distance uncertainties $0.01 \mathrm{m}$ (in the parallel direction of the straight line).

Secondly, the criterion of the optimization problem in \eqref{op} in the SKF algorithm is the trace of a shape matrix. There are several minimum criterions to get one optimal ellipsoid given a shape matrix $S$, e.g., the trace of the shape matrix $\mathrm{tr}(S)$, the determinant of the shape matrix $|S|$, and the largest eigenvalue of the shape matrix $\lambda_M(S)$. Minimizing the largest eigenvalue $\lambda_M(S)$ smoothes the mean curvature and makes the ellipsoid more like a ball (circular in 2D case). Minimizing the trace or the determinant of the shape matrix produces an ellipsoid with small volume, but sometimes causes the ellipsoid getting oblate, i.e., more uncertainties in one certain direction in this example.

\begin{figure}[H]
\centering
\begin{subfigure}[b]{.49\linewidth}
\includegraphics[width=\linewidth]{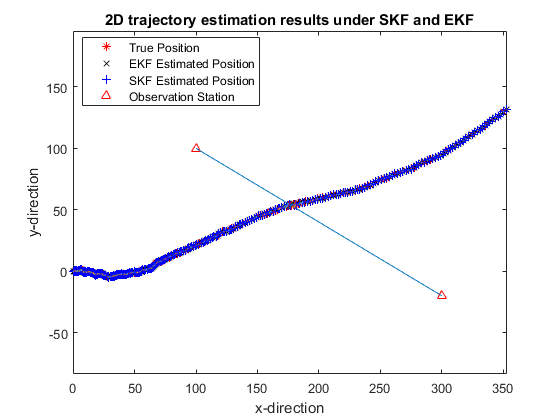}
%\caption*{Su=diag([$1^2,1^2,0.5^2,0.5^2$]),Sz=diag([$0.01^2,0.01^2,(\pi/180)^2,(\pi/180)^2$]).}
\caption{Trajectory 5.1}
\end{subfigure}
\begin{subfigure}[b]{.49\linewidth}
\includegraphics[width=\linewidth]{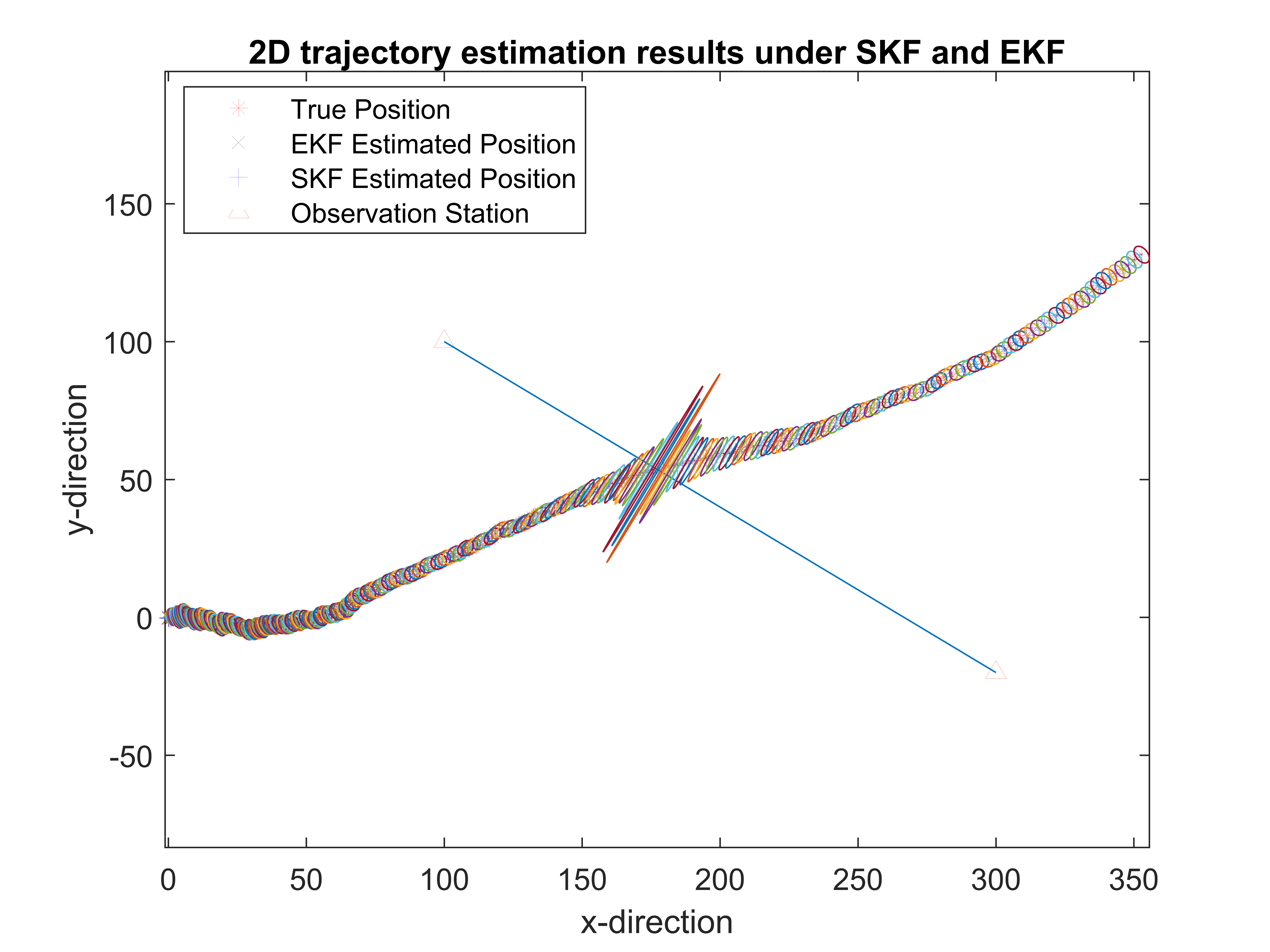}
%\caption*{Su=diag([$10^2,10^2,5^2,5^2$]),Sz=diag([$0.1^2,0.1^2,(\pi/18)^2,(\pi/18)^2$]).}
\caption{Trajectory 5.2}
\end{subfigure}

\begin{subfigure}[b]{.49\linewidth}
\includegraphics[width=\linewidth]{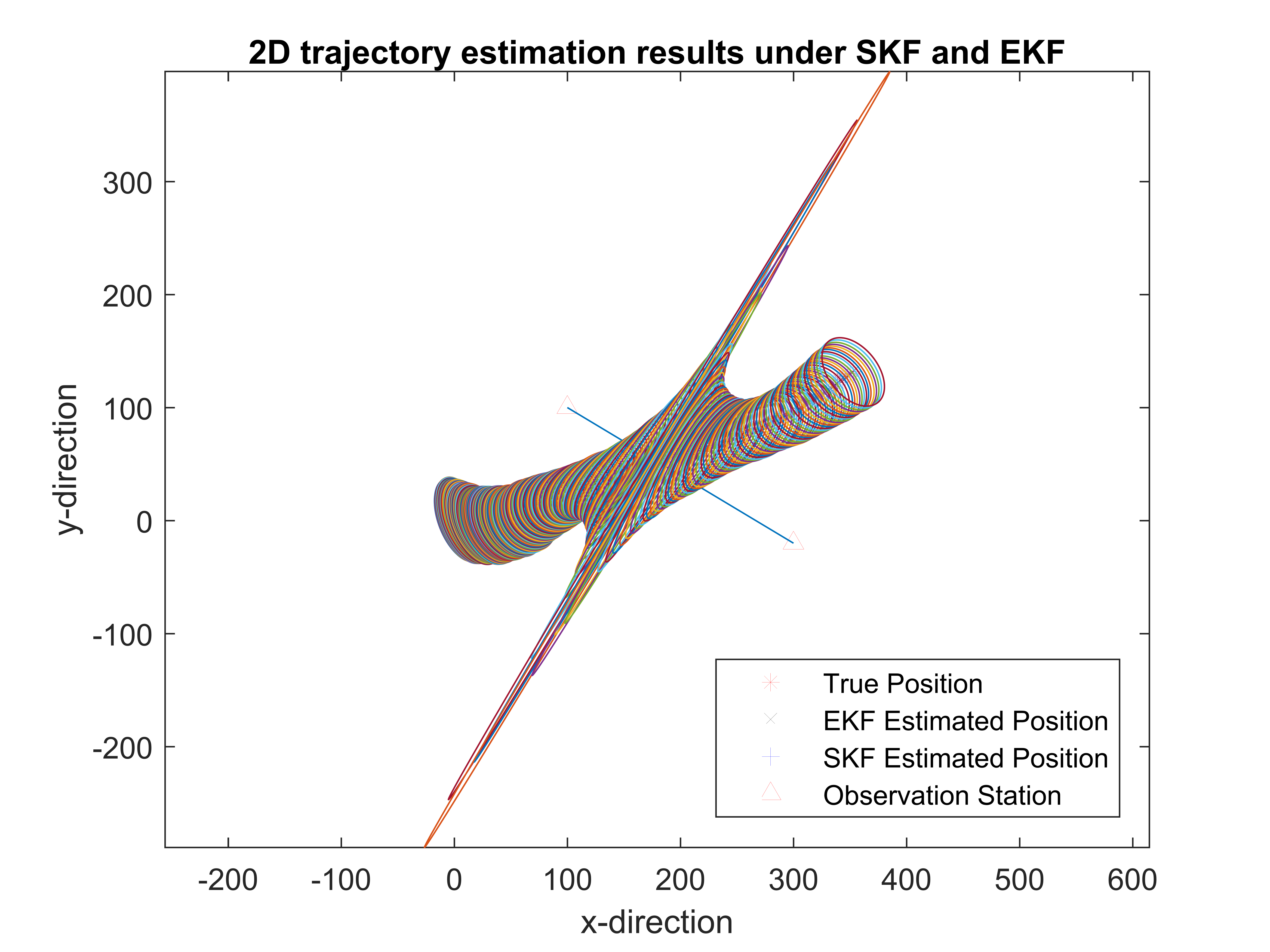}
%\caption*{Su=diag([$100^2,100^2,50^2,50^2$]),Sz=diag([$1^2,1^2,(\pi/1.8)^2,(\pi/1.8)^2$]).}
\caption{Trajectory 5.3}
\end{subfigure}
\begin{subfigure}[b]{.49\linewidth}
\includegraphics[width=\linewidth]{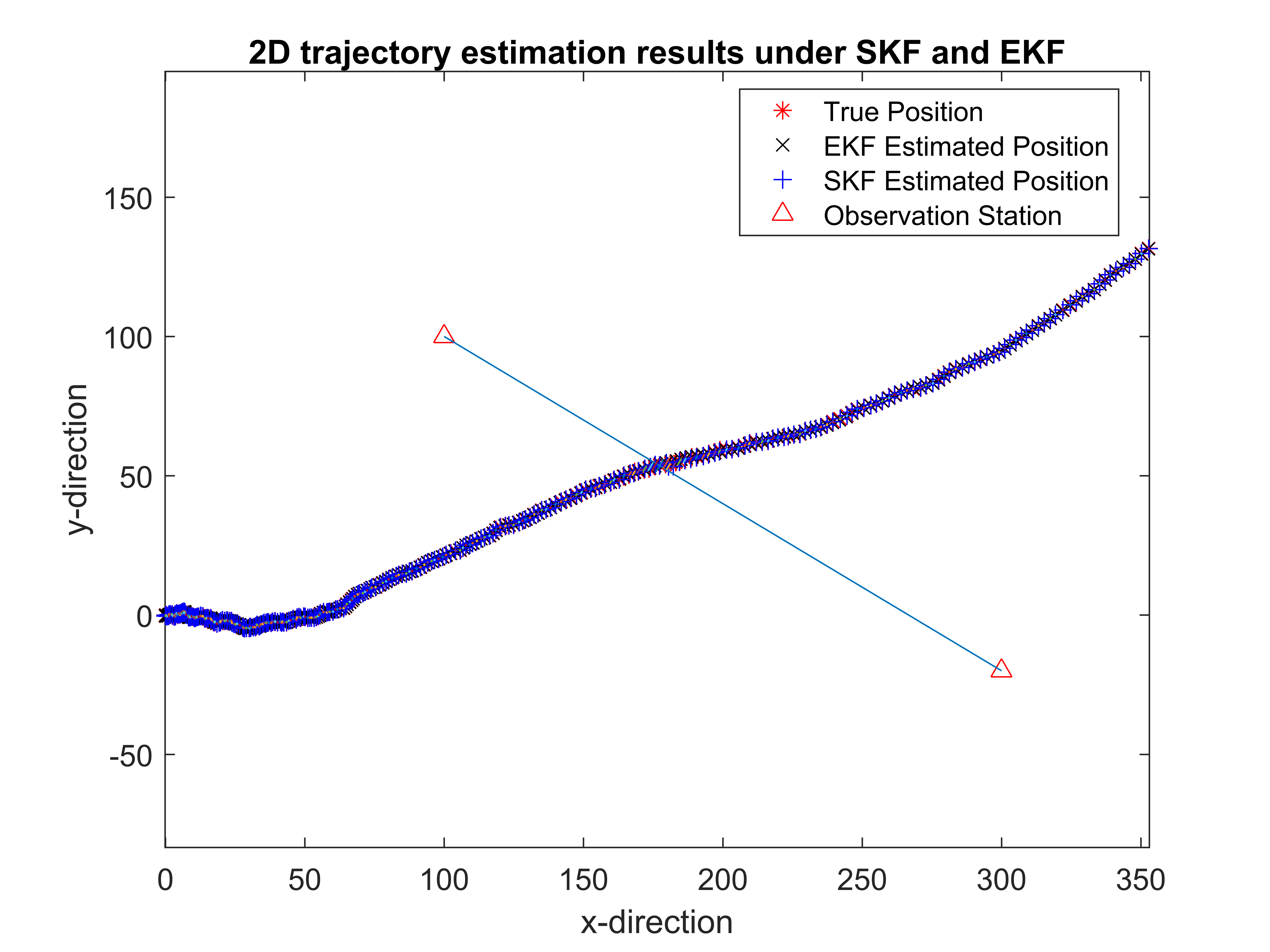}
%\caption*{Su=diag([$0.5^2,0.5^2,0.5^2,0.5^2$]),Sz=diag([$0.01^2,0.01^2,(\pi/180)^2,(\pi/180)^2$]).\\
%Largest Ellipsoid: x-direction: [174.5,177]. y-direction: [52,56].}
\caption{Trajectory 5.4}
\end{subfigure}
\caption{Output ellipsoids are highly related to the initial settings.}
\label{terrible}
\end{figure}

Fig. \ref{terrible} shows that the output estimated ellipsoids are highly related to the initial boundary of the set-membership uncertainties in both equations. The principle semi-axes in Trajectory 5.3 (Su=diag([$100^2,100^2,50^2,50^2$]),\\
Sz=diag([$1^2,1^2,(\pi/1.8)^2,(\pi/1.8)^2$])) is 100 times larger than their counterparts in Trajectory 5.1 (Su=diag([$1^2,1^2,0.5^2,0.5^2$]),\\
Sz=diag([$0.01^2,0.01^2,(\pi/180)^2,(\pi/180)^2$]) and 10 times larger than their counterparts in Trajectory 5.2 (Su=diag([$10^2,10^2,5^2,5^2$]),\\
Sz=diag([$0.1^2,0.1^2,(\pi/18)^2,(\pi/18)^2$])), the the outputs of the SKF are getting very large. Both the input (accuracy of the instruments) and the output (estimated positions) in Trajectory 5.2 and 5.3 are not realistic and one more realistic example was shown in Trajectory 5.4 with initial shape matrix Su=diag([$0.5^2,0.5^2,0.5^2,0.5^2$]),\\
Sz=diag([$0.01^2,0.01^2,(\pi/180)^2,(\pi/180)^2$].

\section{Conclusion and Future Work}

 One cannot state that the new SKF is always better than the standard EKF, however, the performance of SKF is much more reliable than EKF in some cases (like in previous simulated experiments). To say the least, the SKF is one reasonable and applicable model when some unknown but bounded uncertainties were included in the nonlinear system. A difference with the standard Kalman filter is that, the estimated states are ellipsoids instead of single points, and every inner points of one ellipsoid have the same estimation status. But one still can choose a series of particular points in these ellipsoids if necessary. The output is reasonable considering the unknown but bounded uncertainties which were included in the original system, and extra information in the measurement equation was issued properly in the filtering step.

 Like other filter models, there is also some space for this SKF to improve. For instance, the shape matrices of the set-membership uncertainties in both system and measurement equation must be given properly at the beginning, and also the weighting parameter should be decided by the user or experts.

 The future work of our research includes deriving a similar algorithm for second order extend Kalman filter or unscented Kalman filter, using zonotopes or interval boxes to bound the unknown but bounded uncertainty, and minimizing the determinant or the largest eigenvalue of the shape matrix when solving the optimization problem. Last but not least, the stability of this algorithm should be carefully discussed considering that the state estimation problem is usually ill-posed as an inverse problem \cite{blank2007state}.

\section*{Acknowledgements}
This work was supported by the German Research Foundation (DFG) as part of the Research Training Group i.c.sens (RTG 2159). The authors acknowledge Prof. Steffen Sch\"on, Prof. Franz Rottensteiner and Prof. Claus Brenner from Leibniz University for the comments which greatly assisted the research. We also wish to thank Dr. Sergey Grigorian from Department of Mathematics in University of Texas Rio Grande Valley for his valuable suggestions regarding this paper's several aspects.

%\section*{References}

\bibliography{mybibfile}

\begin{thebibliography}{10}
\expandafter\ifx\csname url\endcsname\relax
  \def\url#1{\texttt{#1}}\fi
\expandafter\ifx\csname urlprefix\endcsname\relax\def\urlprefix{URL }\fi
\expandafter\ifx\csname href\endcsname\relax
  \def\href#1#2{#2} \def\path#1{#1}\fi

\bibitem{le2013zonotopic}
V.~T.~H. Le, C.~Stoica, T.~Alamo, E.~F. Camacho, D.~Dumur, Zonotopic guaranteed
  state estimation for uncertain systems, Automatica 49~(11) (2013) 3418--3424.

\bibitem{kalman1960new}
R.~E. Kalman, et~al., A new approach to linear filtering and prediction
  problems, Journal of basic Engineering 82~(1) (1960) 35--45.

\bibitem{smith1962application}
G.~L. Smith, S.~F. Schmidt, L.~A. McGee, Application of statistical filter
  theory to the optimal estimation of position and velocity on board a
  circumlunar vehicle, National Aeronautics and Space Administration, 1962.

\bibitem{mcelhoe1966assessment}
B.~A. McElhoe, An assessment of the navigation and course corrections for a
  manned flyby of mars or venus, IEEE Transactions on Aerospace and Electronic
  Systems~(4) (1966) 613--623.

\bibitem{milanese1996optimal}
M.~Milanese, A.~Vicino, Optimal estimation theory for dynamic systems with set
  membership uncertainty: An overview, in: Bounding Approaches to System
  Identification, Springer, 1996, pp. 5--27.

\bibitem{witsenhausen1968sets}
H.~Witsenhausen, Sets of possible states of linear systems given perturbed
  observations, IEEE Transactions on Automatic Control 13~(5) (1968) 556--558.

\bibitem{schweppe1968recursive}
F.~Schweppe, Recursive state estimation: Unknown but bounded errors and system
  inputs, IEEE Transactions on Automatic Control 13~(1) (1968) 22--28.

\bibitem{schweppe1973uncertainty}
F.~Schweppe, Uncertain dynamic systems, Prentice Halls, Englewood Cliffs NJ,
  1973.

\bibitem{vicino1996sequential}
A.~Vicino, G.~Zappa, Sequential approximation of feasible parameter sets for
  identification with set membership uncertainty, IEEE Transactions on
  Automatic Control 41~(6) (1996) 774--785.

\bibitem{walter1989exact}
E.~Walter, H.~Piet-Lahanier, Exact recursive polyhedral description of the
  feasible parameter set for bounded-error models, IEEE Transactions on
  Automatic Control 34~(8) (1989) 911--915.

\bibitem{bertsekas1971minimax}
D.~P. Bertsekas, I.~B. Rhodes, On the minimax reachability of target sets and
  target tubes, Automatica 7~(2) (1971) 233--247.

\bibitem{polyak2004ellipsoidal}
B.~T. Polyak, S.~A. Nazin, C.~Durieu, E.~Walter, Ellipsoidal parameter or state
  estimation under model uncertainty, Automatica 40~(7) (2004) 1171--1179.

\bibitem{durieu2001multi}
C.~Durieu, E.~Walter, B.~Polyak, Multi-input multi-output ellipsoidal state
  bounding, Journal of optimization theory and applications 111~(2) (2001)
  273--303.

\bibitem{combastel2005state}
C.~Combastel, A state bounding observer for uncertain non-linear
  continuous-time systems based on zonotopes, in: Decision and Control, 2005
  and 2005 European Control Conference. CDC-ECC'05. 44th IEEE Conference on,
  IEEE, 2005, pp. 7228--7234.

\bibitem{alamo2005guaranteed}
T.~Alamo, J.~M. Bravo, E.~F. Camacho, Guaranteed state estimation by zonotopes,
  Automatica 41~(6) (2005) 1035--1043.

\bibitem{althoff2009safety}
M.~Althoff, O.~Stursberg, M.~Buss, Safety assessment for stochastic linear
  systems using enclosing hulls of probability density functions, in: Control
  Conference (ECC), 2009 European, IEEE, 2009, pp. 625--630.

\bibitem{schon2005using}
S.~Sch{\"o}n, H.~Kutterer, Using zonotopes for overestimation-free interval
  least-squares--some geodetic applications, Reliable Computing 11~(2) (2005)
  137--155.

\bibitem{kreinovich2013computational}
V.~Kreinovich, A.~V. Lakeyev, J.~Rohn, P.~Kahl, Computational complexity and
  feasibility of data processing and interval computations, Vol.~10, Springer
  Science \& Business Media, 2013.

\bibitem{ferson2007experimental}
S.~Ferson, V.~Kreinovich, J.~Hajagos, W.~Oberkampf, L.~Ginzburg, Experimental
  uncertainty estimation and statistics for data having interval uncertainty,
  Sandia National Laboratories, Report SAND2007-0939.

\bibitem{kutterer2011recursive}
H.~Kutterer, I.~Neumann, Recursive least-squares estimation in case of interval
  observation data, International Journal of Reliability and Safety 5~(3-4)
  (2011) 229--249.

\bibitem{noack2014state}
B.~Noack, State estimation for distributed systems with stochastic and
  set-membership uncertainties, Vol.~14, KIT Scientific Publishing, 2014.

\bibitem{Ru99a}
S.~Rump, {INTLAB - INTerval LABoratory}, in: T.~Csendes (Ed.),
  {Developments~in~Reliable Computing}, Kluwer Academic Publishers, Dordrecht,
  1999, pp. 77--104.

\bibitem{simon2006optimal}
D.~Simon, Optimal state estimation: Kalman, H infinity, and nonlinear
  approaches, John Wiley \& Sons, 2006.

\bibitem{kitagawa1987non}
G.~Kitagawa, Non-gaussian state—space modeling of nonstationary time series,
  Journal of the American statistical association 82~(400) (1987) 1032--1041.

\bibitem{gordon1993novel}
N.~J. Gordon, D.~J. Salmond, A.~F. Smith, Novel approach to
  nonlinear/non-gaussian bayesian state estimation, in: IEE Proceedings F
  (Radar and Signal Processing), Vol. 140, IET, 1993, pp. 107--113.

\bibitem{alkhatib2008comparison}
H.~Alkhatib, I.~Neumann, H.~Neuner, H.~Kutterer, Comparison of sequential monte
  carlo filtering with kalman filtering for nonlinear state estimation, in: 1st
  International Conference on Machine Control Guidance, 2008, pp. 1--11.

\bibitem{blank2007state}
L.~Blank, State estimation analysed as inverse problem, Assessment and Future
  Directions of Nonlinear Model Predictive Control (2007) 335--346.

\bibitem{candy2016bayesian}
J.~V. Candy, Bayesian signal processing: classical, modern, and particle
  filtering methods, Vol.~54, John Wiley \& Sons, 2016.

\bibitem{klir2005uncertainty}
G.~J. Klir, Uncertainty and information: foundations of generalized information
  theory, John Wiley \& Sons, 2005.

\bibitem{li2018kalman}
S.~E. Li, G.~Li, J.~Yu, C.~Liu, B.~Cheng, J.~Wang, K.~Li, Kalman filter-based
  tracking of moving objects using linear ultrasonic sensor array for road
  vehicles, Mechanical Systems and Signal Processing 98 (2018) 173--189.

\bibitem{sarkka2006recursive}
S.~S{\"a}rkk{\"a}, et~al., Recursive Bayesian inference on stochastic
  differential equations, Helsinki University of Technology, 2006.

\bibitem{kurzhanskiui1997ellipsoidal}
A.~Kurzhanski\u{i}, I.~V\'alyi, Ellipsoidal calculus for estimation and
  control, Nelson Thornes, 1997.

\end{thebibliography}
\nocite{candy2016bayesian}
\nocite{klir2005uncertainty}
\nocite{li2018kalman}
\nocite{sarkka2006recursive}
\nocite{kurzhanskiui1997ellipsoidal}

\end{document}